\documentclass[11pt,reqno,oneside]{amsart}
\usepackage{graphicx, subfigure,}
\usepackage{amssymb}
\usepackage{epstopdf}
\usepackage{amssymb}
\usepackage[a4paper, total={6in, 9.1in}]{geometry}
\usepackage{multirow}
\usepackage{bm}
\usepackage{amsmath,amsfonts,amsthm,mathrsfs,amssymb,cite}
\usepackage{mathtools}
\usepackage{framed}
\usepackage{exscale}
\usepackage{relsize}
\usepackage{enumerate}
\usepackage{epstopdf}
\usepackage{latexsym}
\usepackage[usenames]{color}

\usepackage{graphics}
\usepackage{framed}
\usepackage{enumerate}
\usepackage{hyperref}
\usepackage{xcolor}

\hypersetup{
	colorlinks,
	linkcolor={blue!80!black},
	urlcolor={blue!80!black},
	citecolor={blue!80!black}
}
\usepackage[capitalize,noabbrev]{cleveref}

\numberwithin{equation}{section}

\newtheorem{definition}{Definition}[section]
\newtheorem{theorem}{Theorem}[section]
\newtheorem{lemma}[theorem]{Lemma}

\theoremstyle{definition}

\theoremstyle{remark}
\newtheorem{remark}[theorem]{Remark}

\numberwithin{equation}{section}

\newcounter{saveeqn}

\title{Inverse Problems for Nonlinear Progressive Waves}

\author{Yan Jiang}
\address{Department of Mathematics, Jilin University, Changchun, Jilin, China.}
\email{jiangyan20@mails.jlu.edu.cn}

\author{Hongyu Liu}
\address{Department of Mathematics, City University of Hong Kong, Hong Kong SAR, China.}
\email{hongyu.liuip@gmail.com, hongyliu@cityu.edu.hk}

\author{Tianhao Ni}
\address{School of Mathematical Sciences, Zhejiang University, Hangzhou, Zhejiang, China.}
\email{12035009@zju.edu.cn}

\author{Kai Zhang}
\address{Department of Mathematics, Jilin University, Changchun, Jilin, China.}
\email{zhangkaimath@jlu.edu.cn}

\date{} 
\begin{document}
	
\begin{abstract}

We propose and study several inverse problems associated with the nonlinear progressive waves that arise in infrasonic inversions. The nonlinear progressive equation (NPE) is of a quasilinear form $\partial_t^2 u=\Delta f(x, u)$ with $f(x, u)=c_1(x) u+c_2(x) u^n$, $n\geq 2$, and can be derived from the
hyperbolic system of conservation laws associated with the Euler equations. We establish unique identifiability results in determining $f(x, u)$ as well as the associated initial data by the boundary measurement. Our analysis relies on high-order linearisation and construction of proper Gaussian beam solutions for the underlying wave equations. In addition to its theoretical interest, we connect our study to applications of practical importance in infrasound waveform inversion.

\medskip
	

\end{abstract}
\maketitle

\section{Introduction}\label{sec:intro}

Infrasound refers to sound waves with frequencies below the audible range of human hearing. One of the earliest documented experiments with infrasound is due to Leon Scott who developed the phonoautograph to trace sound waves onto a moving surface in 1857. With the development of sensitive instruments that could detect low-frequency sound waves, infrasonic waves generated by earthquakes, volcanoes, and other natural phenomena can nowadays be effectively measured; see e.g. \cite{MW05,PHMV05,WOH05} and the references cited therein. In recent years, infrasounds continue to attract significant attention in the literature due to their key role in a variety of emerging new applications including the understanding of animal communications and behaviours, and the detection of nuclear explosions as well as other military applications. These infrasonic wave signals can be used to locate the source areas and properties and infer the medium environments (cf. \cite{BBGHOW06}).

The nonlinear progressive-wave equation (NPE) is often used to model the infrasonic wave propagation \cite{MK87,CGRKM04}. Originally designed for the study of nonlinear propagation under water from high-intensity sources, the NPE has proven to be a valuable modeling tool, demonstrating strong agreement with measurements across a wide range of underwater scenarios. Its unique feature lies in its ability to reduce the complex system of Euler fluid dynamical system into a single equation associated with a single variable, namely the density. In \cite{MCW94,CW95}, McDonald et al. developed the NPE model with or without the perturbation expansions, based on the assumption that all propagation is outgoing and the shocks are weak. The core principle of the NPE is the resolution of a simplified nonlinear wave equation within a moving window that surrounds the wavefront. Due to this approach, the method is only suitable for finite-length signals. However, by limiting the calculation domain to a small area around the signal, the computational cost is significantly reduced in comparison to methods based on Euler's equations. To more closely depict the real physical processes, Too and Ginsberg incorporated the thermal viscous absorption effects into the model \cite{TG92a,TG92b}. The corresponding NPE is derived from the motion equation, the continuity equation, and the state equation, with the goal of describing the propagation characteristics of infrasonic waves in non-uniform atmospheres. This is particularly important for low-frequency sound waves, which may exhibit anomalous propagation in the presence of atmospheric wind fields and specific temperature structures. The NPE has been applied successfully in many practical problems including underwater explosions, non-proliferation research, and atmospheric sonic booms. We refer to \cite{LJDS09,McD00} and the references therein for more related background discussion.

In its typical formulation, the NPE is a quasilinear hyperbolic equation of the form
\begin{equation}\label{eq:NPE1}
\partial_t^2 u(x, t)=\Delta f(x, u(x,t)),\quad f(x, u(x, t))=c_1(x) u(x, t)+c_2(x) u^n(x, t),
\end{equation}
where $n\in\mathbb{N}$ and $n\geq 2$. In \eqref{eq:NPE1}, $u(x, t)$ with $(x, t)\in\mathbb{R}^d\times\mathbb{R}_+$ denotes the density function, and $c_j$, $j=1,2$, signify the medium properties of the space. The NPE \eqref{eq:NPE1} is associated with proper initial conditions which signify the infrasound sources. In the next section, we shall provide a more detailed derivation of the NPE from the hyperbolic system of conservation laws associated with the Euler equations. In this paper, we are mainly concerned with the mathematical study of the inverse problems of recovering $c_1$ and $c_2$, and/or the underlying sources by the corresponding boundary measurements. In the physical setup, the inverse problems correspond to identifying the unknown sources of the infrasound and/or the properties of the medium that it penetrate through via boundary measurements of the infrasound wave signals. We establish the unique identifiability results for the inverse problems mentioned above under generic situations. To our best knowledge, these are the first theoretical results in the literature for the inverse boundary problems associated with the NPE \eqref{eq:NPE1}, though they have been extensively studied in the physical and engineering literature. In fact, even from a mathematical point of view, the inverse boundary problem for \eqref{eq:NPE1} presents sufficient technical challenges due to that the unknowns appear in the leading-oder nonlinear term, namely $f(x, u)$. In order to tackle the inverse problems, we develop new techniques that combine the high-order linearisation and the construction of proper Gaussian beams as the ``probing modes". Finally, we would like to mention in passing that a series of inverse problems associated to nonlinear wave equations have been addressed in different mathematical and physical contexts  \cite{LLL23,KLU18,HUW20, UZ21a,UZ21b,WZ19}. Some of those studies are related to the current work and we shall provide more detailed discussion in what follows.

The rest of the paper are organized as follows. In Section~\ref{sec:2}, we present the mathematical setups of the forward and inverse problems as well as state the main results. Section~\ref{sec:3} is devoted to the proof of the main theorem for the local well-posedness of the forward problem, whereas Sections~\ref{sec:4} and \ref{sec:5} are respectively devoted to the proofs of two results regarding the inverse problems.

\section{Mathematical setups and statement of main results}\label{sec:2}

In this section, we present the mathematical setups of the forward and inverse problems as well as state the main results.

\subsection{Nonlinear Progressive Equation}

We briefly discuss the derivation of the NPE from the hyperbolic system of conservation laws by following the approach in \cite{MK87,Gin21}. On the one hand, this provides more details about the physical background and motivation of the current study and on the other hand, it paves the way for many subsequent extensions in our forthcoming works. Let us consider the following Euler equations in three-space dimensions, which describes the conservation of mass momentum and energy:
\begin{equation}\label{eq:Euler}
\left\{
\begin{aligned}
&\rho\left[\frac{\partial v}{\partial t}+(v\cdot \nabla v)\right]=-\nabla p\quad&&\mbox{(Euler's Equation)},\\
&\frac{\partial \rho}{\partial t}+\nabla \cdot (\rho v)=0\quad&&\mbox{(Equation of Continuity)},\\
&p = p(\rho,s)\quad&&\mbox{(State Equation)},
\end{aligned}
\right.
\end{equation}
where $p,\rho,v,s$  are the sound pressure, atmospheric density, velocity field, and entropy, respectively. Writing $v=(v_j)_{j=1}^3$, \eqref{eq:Euler} can be recast in the following form of conservation laws:
\begin{equation}\label{eq:cv1}
\partial_t U(x,t)+\nabla\cdot F(U(x,t))=0,
\end{equation}
where
\begin{equation}
U=\left(
\begin{array}{ll}
\rho\\
\rho v_1\\
\rho v_2\\
\rho v_3
\end{array}
\right),\qquad
F(U)=\left(
\begin{array}{lll}
\rho v_1&\rho v_2&\rho v_3\\
\rho v_1^2+p&\rho v_1v_2&\rho v_1v_3\\
\rho v_2v_1&\rho v_2^2+p&\rho v_2v_3\\
\rho v_3v_1&\rho v_3v_2&\rho v_3^2+p
\end{array}
\right).
\end{equation}

Consider the pulse-following coordinate system $(t',x_1',x_2',x_3')$:
\begin{align*}
t = \epsilon t', \quad
x_1 = x_1'-c_0t',\quad
x_2 = \epsilon^{1/2} x_2',\quad
x_3 = \epsilon^{1/2} x_3',
\end{align*}
where $\epsilon\in\mathbb{R}_+$ signifies an asymptotically small perturbation and $c_0$ is defined in the following Taylor expansion of the state equation
\begin{align*}
	p= c_0^2\rho+\frac{1}{2}\rho^2\frac{\partial^2 p}{\partial \rho^2},
\end{align*}
i.e. $c_0^2:=\frac{\partial p}{\partial \rho}$. Suppose that the pressure, the density, and the velocity
are all small disturbances of constant states such that
\[
w(t',x_1',x_2',x_3') = \sum\limits_{k=0}^{\infty}\epsilon^k w_k(t,x_1,x_2,x_3),\ \ w=\rho, p\ \mbox{or}\ v.
\]
where $\rho_0$ and $p_0$ are constant and $v_0=0$. Then the governing equations in \eqref{eq:Euler} can be rewritten as
\begin{align}
	\label{eq:NPE0}
	\partial_{t'}^2\rho = \nabla'^2 p+\sum\limits_{i',j'=1}^n\partial_{x_i'}\partial_{x_j'}(\rho v_iv_j),
\end{align}
where $v_i$ denotes the $i$-th direction of the velocity and
\begin{eqnarray*}
&\partial_{x_1'} = \partial_{x_1},\quad  &(\partial_{x_2'},\partial_{x_3'})= \epsilon^{1/2}(\partial_{x_2},\partial_{x_3}),\\
&\partial_{t'} = \epsilon\partial_t-c\partial_{x_1},\quad &\nabla'^2 = \partial_{x_1}^2+\epsilon(\partial_{x_2}^2+\partial_{x_3}^2).
\end{eqnarray*}
Supposing further that the first-order approximation for plane waves propagating at a small angle in the $x_1$ direction,
i.e.
\begin{align}
	v = (c_0\rho_1/\rho_0,0,0)^T,
\end{align}
then the equation (\ref{eq:NPE0}) can be rewritten as
\begin{align}
\frac{\partial^2\rho_1}{\partial t^2} = \Delta \left[c_0^2\left(\rho_1+\frac{\beta \rho_1^2}{\rho_0}\right)\right],
\end{align}
where $\beta = 1+\frac{\rho}{c_0}\frac{\partial c_0}{\partial \rho}$ is a nonlinear coefficient.
Let $u=\rho_1$, $c_1=c_0^2$, $c_2=c_0^2\frac{\beta}{\rho_0}$, the NPE \eqref{eq:NPE0} can be recast as
\begin{equation}\label{eq:NPEE1}
\partial_t^2 u= \Delta (c_1u+c_2u^2)
\end{equation}
which is \eqref{eq:NPE1} with $n=2$.

It is noted that in the one-space dimensional case, \eqref{eq:NPEE1} can also be written into the form of conservation laws:
\begin{equation}\label{eq:cv2}
\partial_t V(x,t)+ \nabla\cdot G(V(x,t))=0,
\end{equation}
with
\begin{equation}
V=\left(
\begin{array}{ll}
\rho_1\\
v
\end{array}
\right),\qquad
G(V)=\left(
\begin{array}{ll}
v\\
c_0^2\left(\rho_1+\frac{\beta \rho_1^2}{\rho_0}\right)
\end{array}
\right).
\end{equation}
From \eqref{eq:cv2}, one can consider a more general form of $G(V)=[v, p(x, \rho_1)]^T$ where $p$ is a certain function of $p$, say e.g. $p(x, \rho_1)=\alpha_1(x)\rho_1+\alpha_2(x)\rho_1^n$ with $n\geq 2$, and $\alpha_j$ being the coefficient functions. This more general model can occur in modelling the gas dynamics; see e.g. \cite{Maj}.

\subsection{Mathematical setups of the forward and inverse problems}

Following our discussion in the previous section, we consider the NPE system as follows:
\begin{equation}
\label{eq:NPE order2}
 \left\{
\begin{aligned}
	&\partial_t^2 u= \Delta (c_1u+c_2u^n)\ &&\text{in}\ M:=[0,T]\times\Omega,\\
	&u(0,x) = \varphi(x),\ u_t(0,x) = \psi(x)\ &&\text{on}\ \Omega,\\
    &u(t,x) = h(t,x)\ &&\text{in}\ \Sigma:=[0,T]\times \Gamma,		
\end{aligned}\right.
\end{equation}
where $n\in\mathbb{N}$ and $n\geq 2$, $h\in L^2(\Sigma)$ and $(\varphi,\psi)\in H^1(\Omega)\times L^2(\Omega)$. Throughout the rest of the paper, we confine our study within the three-space dimensions. In \eqref{eq:NPE order2}, $\Omega\subset \mathbb{R}^3$ is a nonempty bounded domain with a smooth boundary {$\Gamma=\partial \Omega \in C^{\infty}$} and $T>0$. We let $\nu$ denote the unit outer normal vector on $\Gamma$. In the physical setup, $\varphi$ and $\psi$ signify the sources, whereas $h$ denotes the boundary input. Furthermore, for the physical purpose, the functions $c_1$ and $c_2$ satisfy the following admissible coefficient conditions:
\begin{enumerate}
	\item[(1)] $c_1$ and $c_2$ are only dependent on the spatial variable $x$ and smooth with respect to $x$;\smallskip
	
	\item[(2)] There exist {two constants $M_1$ and $M_2$ such that $M_2>c_i(x)\geq M_1>0$, $i=1,2$};\smallskip
	
    \item[(3)] $\nu\cdot \nabla c_1=0$ on $\Gamma$.
\end{enumerate}

Next, we introduce some function spaces and notations for our subsequent use. For any integer $m\geq 0$, we set
\begin{equation*}
	\mathcal{O}^{m}(\Sigma) = \left\{h \in H^m(\Sigma);\,  h\in H_0^{m-k}([0,T];H^k(\Gamma)),\, \forall\, k=0,1,\cdots,m-1\right\},
\end{equation*}
and define the energy space as
\begin{equation*}
	E^m = \bigcap\limits_{k=0}^m C^k([0,T];H^{m-k}(\Omega)),
\end{equation*}
where the associated norm of the energy space $E^m$ is given by
\begin{equation*}
	\|v\|_{E^m} = \sup\limits_{0\leq t\leq T}\sum\limits_{k=0}^m\left\|\partial_t^k v(t,\cdot)\right\|_{H^{m-k}(\Omega)},\ \forall\, v\in E^m.
\end{equation*}
Moreover, it follows from the estimate	
\begin{equation*}
	\|wv\|_{E^m}\leq C_m\|w\|_{E^m}\|v\|_{E^m},\ \forall\, w,v \in E^m,
\end{equation*}
that the space $E^m$ is a Banach algebra (cf.\cite{Bru09}).

Associated with the NPE system \eqref{eq:NPE order2}, we set
\begin{equation}\label{eq:DL}
	D^L(u)=c_1\nabla u,\quad D^N(u)=c_1 \nabla (\frac{c_2}{c_1^n}u^n),
\end{equation}
and
\begin{equation}\label{eq:L}
	L(u)=\nabla\cdot D^L(u),\quad N(u)=\nabla\cdot  D^{N}(u).
\end{equation}
Then we introduce the following Dirichlet-to-Neumann (NtD) map:
\begin{equation}\label{eq:kappa def}
\Lambda_{\varphi,\psi,L,N} (h)=\nu\cdot \left[D^L(u)+D^N(u)\right]\big|_{\Sigma},
\end{equation}
where $h=u\big|_{\Sigma}\in \mathcal{O}^{m+1}(\Sigma)$ and $u\in E^{m+1}$ is the solution to \eqref{eq:NPE order2}. In Section~\ref{sec:3}, we shall prove the well-posedness of the DtN map $\Lambda_{\varphi,\psi, L, N}$ in a proper sense. The DtN map encodes the measurement data on $\partial\Omega$ associated with different probing inputs $h$ on $\partial\Omega$, as well as the underlying target objects including the medium configurations $c_1, c_2$ and the infrasound sources $\varphi, \psi$. In the case that $\varphi=\psi=0$, the infrasound field is solely determined by the interaction between $c_1, c_2$ and the probing input $h$, and we abbreviate $\Lambda_{0,0,L,N}$ as $\Lambda_{L,N}$. The first inverse problem that we are concerned with is to recover $c_1$ and $c_2$ by knowledge of $\Lambda_{L,N}$, assuming that $\varphi=\psi\equiv 0$. That is,
\begin{equation}\label{eq:IP1A}
(\mbox{IP1})\qquad\qquad\Lambda_{L,N}\longrightarrow L\ \mbox{and}\ N.
\end{equation}
We shall mainly address the theoretical unique identifiability issue for the above inverse problem. That is, we shall derive general conditions on the unknown $c_1$ and $c_2$ which guarantee that $\Lambda_{L_1, N_1}=\Lambda_{L_2, N_2}$ if and only if $(L_1, N_1)=(L_2, N_2)$, where $(L_j, N_j)$, $j=1,2$ are two admissible medium configurations. Here, we note that if letting $(c_1^{(j)}, c_2^{(j)})$ denote the underlying medium parameters associated with $(L_j, N_j)$, $j=1,2$, the above uniqueness result clearly yields that $(c_1^{(1)}, c_2^{(1)})=(c_1^{(2)}, c_2^{(2)})$. The second inverse problem of our study is to simultaneously recover the underlying sources and the medium configuration by knowledge of $\Lambda_{\varphi, \psi, L, N}$:
\begin{equation}\label{eq:IP1B}
(\mbox{IP2})\qquad\qquad \Lambda_{\varphi, \psi, L, N}\longrightarrow \varphi, \psi, L, N.
\end{equation}
In order to establish the unique identifiability for IP2, we need to assume that $L$ is a-priori known, which shall be detailed in what follows. Finally, we would also like to note that the cases $n=2$ and $n>2$ in \eqref{eq:NPE order2} shall present some technical differences for the inverse problem study, which shall also be detailed in what follows.


\subsection{Statement of the main results on direct and inverse NPE problems}

The first result is on the local well-posedness of the NPE system \eqref{eq:NPE order2} which paves the way for our study of the inverse problems IP1 and IP2. To that end, we first introduce the following compatibility conditions on the inputs $\varphi, \psi$ and $h$ in \eqref{eq:NPE order2}.

\begin{definition} On the boundary $\Gamma$ and $m\geq 0$, if the following relations hold
\begin{equation}\label{eq:compatibility conditions}
\left\{\begin{array}{l}
h(x, 0)=\varphi(x) \text { on } \Gamma, \\
h_t(x, 0)=\psi(x) \text { on } \Gamma, \\
h_{tt}(x, 0)=\Delta (c_1(x)\varphi(x)+c_2(x)\varphi^2(x)) \text { on } \Gamma, \\
h_{ttt}(x, 0)=\Delta (c_1(x)\psi(x)+2c_2(x)\varphi(x)\psi(x)) \text { on } \Gamma, \\
\cdots \\
\text{up to the $m$th-order temporal derivative of } $h$ \text { on } \Gamma,
\end{array}\right.
\end{equation}
the triple $(\varphi,\psi,h)$ is said to satisfy the compatibility conditions up to the order $m$.
\end{definition}

Now, we are in a position to show the well-posedness of the forward problem \eqref{eq:NPE order2}.

\begin{theorem}\label{thm:LWP1}
Suppose $m>\frac{3}{2}$, and $c_1$ and $c_2$ satisfy the admissible coefficient conditions. Then there exists a constant $\delta>0$ such that for any $(\varphi,\psi,h)$ belonging to the set
\begin{equation}\label{eq:setn1}
\begin{split}
\mathcal{N}_{\delta} =& \bigg\{(\varphi,\psi,h)\in H_0^{m+1}(\Omega)\times H_0^m(\Omega)\times \mathcal{O}^{m+1}(\Sigma);\\
&\qquad (\varphi, \psi, h)\ \mbox{satisfies the $m$th-order compatibility conditions and}\\
&\qquad  \|\varphi\|_{H^{m+1}(\Omega)}+\|\psi\|_{H^{m}(\Omega)}+\|h\|_{H^{m+1}({\Sigma})}<\delta  \bigg\},
\end{split}
\end{equation}
the forward problem \eqref{eq:NPE order2} admits a unique solution $u\in E^{m+1}$ satisfying that $\partial_{\nu} u\in H^{m}({\Sigma})$ and
\begin{equation}\label{eq:est u P1}
	\|u\|_{E^{m+1}}+\|u\|_{C(M)}+\|\partial_{\nu} u\|_{H^{m}(\Sigma)}\leq C\delta,
\end{equation}
where $C$ is a positive constant depends only on $\Omega$ and $T$.

In addition, the mapping which sends the initial boundary conditions to the solution of \eqref{eq:NPE order2} as:
\begin{equation}\label{eq:def F}
    F:(\varphi,\psi,h)\mapsto u, \quad  \mathcal{N}_{\delta}\to E^{m+1},
\end{equation}
is $C^{\infty}$ Fr\'{e}chet differentiable.
\end{theorem}
\begin{remark}
Theorem~\ref{thm:LWP1} states that the solution to \eqref{eq:NPE order2} with a small norm is well-posed. In fact, as discussed earlier, $u$ in (\ref{eq:NPE order2}) describes the small perturbation of the density, and hence it is physically justifiable to consider small solutions of (\ref{eq:NPE order2}).

It is also worth mentioning that the smoothness of mapping $F$ only depends on that of the coefficients $c_1$ and $c_2$. In fact, if $c_1$ and $c_2\in C^k(\Omega)$, $k\geq 3$, $F$ is $C^k$ Fr\'{e}chet differentiable on $\mathcal{N}_{\delta}$.

Finally, we would like to remark that though we only consider the three-dimensional case, the local well-posedness result in Theorem~\ref{thm:LWP1} holds for any space dimension $d\geq 2$ as long as $m>d/2$. In fact, this can be easily seen in our proof of the theorem in what follows. Nevertheless, we shall stick our study to the 3D case, especially for the relevant inverse problems.
\end{remark}

We proceed to consider the inverse problems associated with \eqref{eq:NPE order2}, i.e. IP1 and IP2. In studying the inverse problems, we need to utilize geodesics on a Lorentz manifold associated with $c_1$ and $c_2$. Hence, we next introduce the definition of the associated metric on the manifold mentioned above.
Let $\mathrm{d}s^2$ be the Euclidean metric, and the Riemannian metric $g_{c_1}$ associated with the wave speed $c_0$ be
\begin{equation}\label{eq:gc1}
		g_{c_1} = c_0^{-2}\mathrm{d}s^2=c_{1}^{-1}\mathrm{d}s^2.
\end{equation}
Then the wave travels along geodesics in the Riemannian manifolds $(\Omega,g_{c_1})$. Let diam$(\Omega)$ be the diameter of $\Omega$ with respect to $g_{c_1}$, i.e.
\begin{equation}
	\text{diam}(\Omega)=\sup\{\text{length of all geodesics in }(\Omega,g_{c_1})\}.
\end{equation}
We introduce the following two assumptions on $(\Omega, c_1, c_2)$ for the inverse problem study.

\medskip

{\bf Assumption I}. $T>$diam$(\Omega)$, $\partial\Omega$ is strictly convex with respect to the Riemannian metric $g_{c_1}$,
and either one of the following conditions holds (cf.\cite{WZ19})
\begin{itemize}
	\item[(i)] $(\Omega,g_{c_1})$ has no conjugate points;
	\item[(ii)] $(\Omega,g_{c_1})$ satisfies the foliation condition.
\end{itemize}

\medskip

{\bf Assumption II}. Let $\Omega_1$ be the smooth domain in $\mathbb{R}^3$ containing $\overline{\Omega}$. It is assumed that $c_1$ can be analytically continued from $\Omega$ to $\Omega_1$, which is still denoted by $c_1$. Moreover, it holds that
\begin{equation*}
		\exists\,\beta\in(0,2],\ \text{such that}\ x\cdot \nabla \frac{1}{{c_1}(x)}+(2-\beta)\frac{1}{{c_1}(x)}\ge0\ \text{in the distributional sense in}\ \mathscr{D}'(\Omega_1);
\end{equation*}
that is
\begin{equation}
\forall \varphi \in \mathscr{D}\left(\Omega_1\right) \text { with } \varphi \geq 0, \quad \int_{\mathbb{R}^3} \frac{1}{{c_1}(x)}(-\operatorname{div}(x \varphi(x))+(2-\beta) \varphi(x)) \mathrm{d} x \geq 0.
\end{equation}

Assumption I shall be needed in establishing the unique identifiability result for inverse problem IP1, whereas both Assumptions I and II shall be needed in establishing the unique identifiability result for inverse problem IP2. In fact, we have

\begin{theorem}\label{thm:IP1}
Suppose that $\varphi=\psi\equiv 0$ and {\bf Assumption I} holds. If
\begin{equation}
    \Lambda_{L_1, N_1}(h)=\Lambda_{L_2, N_2}(h), \quad  \forall\, (0, 0, h)\in \mathcal{N}_{\delta},
\end{equation}
where $\mathcal{N}_\delta$ is defined in \eqref{eq:setn1}, then one has
\begin{equation}
    (L_1, N_1)=(L_2, N_2).
\end{equation}
That is, the map $\Lambda_{L, N}$ uniquely determines $(L, N)$.
\end{theorem}

\begin{theorem}\label{thm:IP2}
Suppose that {\bf Assumptions I} and {\bf II} hold, and $L$ is a-priori known. If
\begin{equation}
	\Lambda_{\varphi_1, \psi_1, L, N_1}(h)=\Lambda_{\varphi_2, \psi_2, L, N_2}(h),\quad  \forall\, (\varphi,\psi,h)\in \mathcal{N}_{\delta},
\end{equation}
where $\mathcal{N}_\delta$ is defined in \eqref{eq:setn1}, then one has
\begin{equation}
	(\varphi_1, \psi_1, N_1)=(\varphi_2, \psi_2, N_2).
\end{equation}
That is, the map $\Lambda_{\varphi, \psi, L, N}$ with an a-priori given $L$ uniquely determines $(\varphi, \psi, N)$.
\end{theorem}

\section{Proof of Theorem~\ref{thm:LWP1}}\label{sec:3}

In this section, we present the proof of Theorem~\ref{thm:LWP1} on the local well-posedness of the forward NPE system \eqref{eq:NPE order2} by employing the implicit function theorem in Banach spaces.

First, we consider following auxiliary hyperbolic equation:
\begin{equation}\label{eq:linear hyper eqn}
\left\{
\begin{aligned}
	&u_{tt}-a(x,D)u={f}\ &&\text{in}\ M\\
	&u(0,x) = \varphi(x),\ u_t(0,x) = \psi(x)\ &&\text {on}\ \Omega,\\
	&u(t,x) = h(t,x)\ &&\text {in}\ \Sigma,
\end{aligned}
\right.
\end{equation}
where
\begin{equation*}
	a(x,D)v=-\sum\limits_{j,k=1}^{n}a_{2}^{jk}(x)\partial_j\partial_kv(x)+\sum\limits_{j=1}^{n}a_{1}^j(x)\partial_jv(x)+a_{0}(x)v(x)
\end{equation*}
with $\{a_{2}^{jk}\}_{j,k}$ being a symmetric matrix satisfying
\begin{equation*}
	\sum\limits_{j,k}a_{2}^{jk}p_jp_k\ge \gamma \sum\limits_{j}p_j^2,\ \ \gamma>0.
\end{equation*}
	
The global well-posedness of the hyperbolic equation (\ref{eq:linear hyper eqn}) is given by following lemma (cf.\cite[Theorem\,2.45]{KKL01}).
\begin{lemma}\label{lem:wellposedness linear hyper eqn}
Let $m$ be a positive integer and $T>0$. For any $\varphi \in H_0^{m+1}(\Omega)$, $\psi\in H_0^m(\Omega)$, $h\in\mathcal{O}^{m+1}(\Sigma)$, $a_{2}^{jk}$, $a_{1}^j$, $a_0 \in E^m $, $f\in E^m$ with $\partial_t^k {f}(\cdot,0)\in H_0^{m-k}(\Omega)$ for $k=0,1,\cdots,m-2$, if the compatibility conditions hold, the hyperbolic equation (\ref{eq:linear hyper eqn}) admits a unique solution $u\in E^{m+1}$ and $\partial_{\nu} u\in H^m(\Sigma)$.
\end{lemma}
	

With the above lemma, we proceed to prove the local well-posedness of the forward NPE system \eqref{eq:NPE order2}, namely Theorem~\ref{thm:LWP1}. 

\begin{proof}[Proof of Theorem\,\ref{thm:LWP1}]

The proof consists of three steps.

\medskip

\noindent {\bf Step 1.}~Let
\begin{equation*}
\begin{aligned}
&U_0=\left\{v \in E^m;\ \partial_t^k v(0,\cdot)\in H_0^{m-k}(\Omega),\  k=0,1,\cdots,m-2\right\},\\
&U_1 = \Big\{v\in E^{m+1};\ v(0,\cdot)\in H_0^{m+1}(\Omega),\ v_t(0,\cdot)\in H_0^m(\Omega),\ v|_{\Sigma}\in\mathcal{O}^{m+1}(\Sigma),\
\partial_{\nu}v\in H^m(\Sigma), \Big.\\
&\Big. \quad \qquad  (v_{tt}-L(v))\in U_0,\ (v_{tt}-L(v)-N(v))\in U_0 \Big\},\\
&U_2=H_0^{m+1}(\Omega)\times H_0^m(\Omega)\times \mathcal{O}^{m+1}(\Sigma),\\
&U_3 = U_0\times U_2,
\end{aligned}
\end{equation*}
with
\begin{equation}\label{eq:U1 norm}
\|v\|_{U_1}=\|v\|_{E^{m+1}}+\left\|\partial_\nu v\right\|_{H^m(\Sigma)}+\|v\|_{H^{m+1}(\Sigma)}+\left\|v_{t t}-\Delta v\right\|_{E^m},
\end{equation}
and set the error operator as
\begin{equation*}
E:(v;\varphi,\psi,h)\mapsto \left(v_{tt}-L(v)-N(v);v(x,0)-\varphi(x),v_t(x,0)-\psi,v|_{\Sigma}-h\right), \, U_1\times U_2\to U_3,
\end{equation*}
where $v\in U_1$ and $(\varphi,\psi,h) \in U_2$.

In order to demonstrate that the image set of the mapping $E$ is $U_3$, it is sufficient to show that
\begin{equation*}
v_{tt}-L(v)-N(v)\in U_0,
\end{equation*}
which is equivalent to
\begin{equation*}
v_{tt}-L(v)-N(v)\in E^m,\quad \partial_t^k v(0,\cdot)(v_{tt}-L(v)-N(v))\in H_0^{m-k}(\Omega), \quad k=0,1,\cdots,m-2.
\end{equation*}

It is obvious that $v_{tt}-L(v)\in E^m$. On the other hand,
\begin{equation*}
N(v)=\Delta (c_2v^n)=c_2(n(n-1)v^{n-2}\nabla v\cdot\nabla v+nv^{n-1}\Delta v)+v^n\Delta c_2+2nv^{n-1}\nabla c_2\cdot \nabla v.
\end{equation*}
It follows from the algebraic properties of $E^m$ that
\begin{align*}
&\|c_2(n(n-1)v^{n-2}\nabla v\cdot\nabla v+nv^{n-1}\Delta v)\|_{E^m}\leq M(\|v\|^{n-2}_{E^m}\|\nabla v\|^2_{E^m}+\|v\|^{n-1}_{E^m}\|\Delta v\|_{E^m}),\\
&\|v^n\Delta c_2\|_{E^m}\leq M\|v\|^n_{E^m},\\
&\|2nv^{n-1}\nabla c_2\cdot \nabla v\|_{E^m}\leq M\|v\|^{n-1}_{E^m}\|\nabla v\|_{E^m},
\end{align*}
where $M=\max\{2\|c_2\|_{H^m(\Omega)}, \|\Delta c_2\|_{H^m(\Omega)}\}$. These estimates readily imply that $N(v)\in E^m$, and hence $v_{tt}-L(v)-N(v)\in E^m$.
In the same manner, we know that $\partial_t^k v(0,\cdot)\in H_0^{m-k}(\Omega)$ and $v_{tt}-L(v)-N(v)\in U_0$. Therefore the mapping $E$ is well-defined.

Then the equation NPE (\ref{eq:NPE order2}) can be rewritten as: for any $(\varphi,\psi,h)\in U_2$, find $u\in U_1$ such that
$E(u;\varphi,\psi,h)=(0;0,0,0)$.

\medskip

\noindent {\bf Step 2.}~We claim that the error operator $E$ is analytic with respect to each of its arguments.
		
By the result in \cite[p.133]{Pos87}, we only need to show that $E$ is locally bounded and weakly analytic.
The former can be easily verified by the definition, and the latter is equivalent to show that for
any $(u_0;\varphi_0,\psi_0,h_0)$, $(u;\varphi,\psi,h,u)\in U_1\times U_2$, the map
\begin{equation*}
		\lambda\mapsto E((\varphi_0,\psi_0,h_0,u_0)+\lambda (\varphi,\psi,h,u)),
\end{equation*}
is analytic in a neighborhood of the origin in $\mathbb{C}$.

Since any polynomial is analytic, the operators $L$ and $N$ are analytic with respect to the parameter $\lambda$.
This implies $E$ is analytic with respect to $\lambda$, and hence the claim is proved.

\medskip

\noindent {\bf Step 3.}~The Fr\'{e}chet derivative of $E$ with respect to $u$ at zero is defined by
\begin{equation*}
		E_u(0,0,0,0)v = \left(v_{tt}-\Delta (c_1 v); v(0,\cdot), v_t(0,\cdot), v|_{\Sigma}\right), \forall\, v\in U_1.
\end{equation*}
Next, we prove that $E_u$ is a linear isomorphism. That is, for any $f\in E^m$ and $(\varphi,\psi,h)\in U_2$, we can find a unique $w\in U_1$ such that
\begin{equation*}
		E_u(0,0,0,0)w=(f,\varphi,\psi,h),
\end{equation*}
which is equivalent to the following linear hyperbolic equation
\begin{equation}\label{eq:w}
\left\{
\begin{aligned}
	&w_{tt}-\Delta (c_1w)=f\ &&\text{in}\ M,\\
	&w(0,x) = \varphi(x),\ w_t(0,x) = \psi(x)\ &&\text {on}\ \Omega, \\
    &w(t,x) = h(t,x)\ &&\text {in}\ \Sigma.\\
\end{aligned}
\right.
\end{equation}

Let $v=c_1w$. Then it is readily seen that $v$ satisfies
\begin{equation}\label{eq:v}
\left\{
\begin{aligned}
	&v_{tt}-c_1\Delta v=c_1f\ &&\text{in}\ M,\\
	&v(0,x) =c_1(x) \varphi(x),\ v_t(0,x) = c_1(x)\psi(x)\ &&\text {on}\ \Omega,\\
    &v(t,x) = h(t,x)/c_1(x)\ &&\text {in}\ \Sigma.
\end{aligned}
\right.
\end{equation}
The global well-posedness of the equation (\ref{eq:v}) follows form Lemma\,\ref{lem:wellposedness linear hyper eqn},
which implies that $E_u$ is a linear isomorphism.
		
Since $E$ is analytic with respect to each of its arguments and $E_u$ is a linear isomorphism,
it follows from the Implicit Function Theorem in Banach spaces that, there exists a constant $\delta>0$ and a $C^{\infty}$
map $F:\mathcal{N}_{\delta}\to E^{m+1}$ such that for any $(\varphi,\psi ,h)\in U_1$ satisfying
\begin{equation}\label{eq:C delta}
		\|\varphi\|_{H^{m+1}(\Omega)}+\|\psi\|_{H^{m}(\Omega)}+\|h\|_{H^{m+1}(\Sigma)}<\delta,
\end{equation}
one has that $E(F(\varphi,\psi,h); \varphi, \psi, h)=(0; 0, 0, 0)$ means that $u=F(\varphi,\psi,h)$ is a solution of the equation (\ref{eq:NPE order2}).
		
Furthermore, by the fact that $u=F(\varphi,\psi,h)$, $0=F(0,0,0)$, and the local Lipschitz continuity of $F$, we have
\begin{equation*}
\|u\|_{E^{m+1}}=\|F(\varphi,\psi,h)-F(0,0,0)\|_{E^{m+1}}
\leq C(\|\varphi\|_{H^{m+1}(\Omega)}+\|\psi\|_{H^m(\Omega)}+\|h\|_{H^{m+1}(\Sigma)}).
\end{equation*}
Finally, by using the trace theorem, the Sobolev embedding theorem, and (\ref{eq:C delta}), we can obtain the estimate (\ref{eq:est u P1}).

The proof is complete. 
\end{proof}


\section{Proof of Theorem~\ref{thm:IP1}}\label{sec:4}

In this section, we present the proof of Theorem~\ref{thm:IP1}. We first consider the case $n=2$, which corresponds to the canonical NPE problem. i.e. determining the coefficients of the NPE (\ref{eq:NPE order2}) by the measurement on the boundary with the homogeneous initial conditions. We shall distinguish between two cases with $n=2$ and $n> 2$ in \eqref{eq:NPE order2}.
Here, we determine the coefficient $c_1$ in the linear term based on the first order linearization technique.
Then the second order linearization technique and the construction of Gaussian beam solutions are used to determine the coefficient $c_2$ in the nonlinear term.


\subsection{First-order linearization}\label{sec:first-order}

For the purpose of linearizing the solution $u$ of equation (\ref{eq:NPE order2}), we first consider the following boundary data
\begin{equation*}
	h(t,x) = \varepsilon_1h^{(1)}(t,x),
\end{equation*}
where $\varepsilon_1\|h^{(1)}\|_{H^{m+1}(\Sigma)}\leq \delta$ with $\delta$ being defined in Theorem\,\ref{thm:LWP1} and $h^{(1)}\in H^{m+1}(\Sigma)$. Then the associated solution $u(x,t,\varepsilon_1)$ yields
\begin{equation}\label{eq:u eps1}
\left\{
\begin{aligned}
&u_{tt}(t,x,\varepsilon_1)-\Delta \left[c_1u(t,x,\varepsilon_1)+c_2u^2(t,x,\varepsilon_1)\right] =0 \ &&\text{in}\ M,\\
&u(0,x,\varepsilon_1)=u_t(0,x,\varepsilon_1) = 0\ &&\text{in}\ \Omega,\\
&u(t,x,\varepsilon_1)=\varepsilon_1 h^{(1)}\ &&\text{on}\ \Sigma.
\end{aligned}
\right.
\end{equation}	
Write $u(t,x,\varepsilon_1)$ as $u$ for short and set the linearization part of $u$ as
\begin{equation*}
u^{(1)}=\frac{\partial u}{\partial{\varepsilon_1}}\left|_{\varepsilon_1=0}\right.,
\end{equation*}
which satisfies
\begin{equation}\label{eq:uj}
\left\{
\begin{aligned}
&u^{(1)}_{tt} -\Delta (c_1u^{(1)})=0\ &&\text{in}\ M,\\
&u^{(1)}(0,x) = u^{(1)}_t(0,x) = 0\ &&\text{in}\ \Omega,\\
&u^{(1)}(t,x)= h^{(1)}(t,x)\ &&\text{on}\ \Sigma.
\end{aligned}
\right.
\end{equation}
In the same manner to (\ref{eq:kappa def}), the linearization part of the DtN mapping $\Lambda_{L,N}$ for (\ref{eq:u eps1}) is given by
\begin{equation}\label{eq:Lambda u}
\Lambda^{(u,1)}_{L,N}:u\left|_{\Sigma}\right.~\mbox{or}~h^{(1)} \mapsto \frac{\partial\Lambda_{L,N}}{\partial\varepsilon_1}(\varepsilon_1h^{(1)})\left|_{\varepsilon_1=0}\right.
=\nu\cdot D^L(u).
\end{equation}
Let $v^{(1)}=c_1u^{(1)}$. Then (\ref{eq:uj}) and (\ref{eq:Lambda u}) can be rewritten as
\begin{equation}\label{eq:vj}
\left\{
\begin{aligned}
&v^{(1)}_{tt}-c_1\Delta v^{(1)}=0\ &&\text{in}\ M,\\
&v^{(1)}(0,x) = v^{(1)}_t(0,x) = 0\ &&\text{in}\ \Omega,\\
&v^{(1)}(t,x)=c_1 h^{(1)}(t,x)\ &&\text{on}\ \Sigma,
\end{aligned}
\right.
\end{equation}
and
\begin{equation}\label{eq:Lambda v}
\Lambda^{(v,1)}_{L,N}:c_1h^{(1)}=v^{(1)}\left|_{\Sigma}\right.\mapsto \frac{\partial \Lambda_{L,N}}{\partial\varepsilon_1}(\varepsilon_1h^{(1)})\left|_{\varepsilon_1=0}\right.
=c_1\frac{\partial v^{(1)}}{\partial\nu},
\end{equation}
where we use the third condition from the admissible coefficient conditions in the last equality.

Combining (\ref{eq:Lambda u}) and (\ref{eq:Lambda v}), we obtain
\begin{equation*}
\Lambda^{(v,1)}_{L,N}=\Lambda^{(u,1)}_{L,N}\circ c_1^{-1},
\end{equation*}
which together with the fact that $c_1$ satisfies the admissible coefficient conditions readily gives that
\begin{equation}\label{eq:u implies v}
\Lambda^{(u,1)}_{L_1,N}=\Lambda^{(u,1)}_{L_2,N} \Rightarrow \Lambda^{(v,1)}_{L_1,N}=\Lambda^{(v,1)}_{L_2,N}.
\end{equation}

On the other hand, it is known from \cite{Bel87} that $c_1$ is uniquely determined by knowledge of the NtD map. 
Meanwhile, based on the well-posedness of the linear hyperbolic equation, we know that the DtN map and NtD map are equivalent.
This implies the uniqueness of determining $c_1$ by using the DtN mapping $\Lambda^{(v,1)}_{L,N}$, as well as the DtN mapping $\Lambda^{(u,1)}_{L,N}$.


\subsection{Second order linearization}

In the previous subsection, we use the first order linearization and the DtN map to recvoer $c_1$. Now we shall make use the
second order linearization and the DtN map to determine the coefficient $c_2$ from the nonlinear term.

Consider the following boundary data
\begin{equation*}
	h(t,x) = \varepsilon_1 h^{(1)}+\varepsilon_2 h^{(2)},
\end{equation*}
where $\varepsilon_1\|h^{(1)}\|_{H^{m+1}(\Sigma)}+\varepsilon_2\|h^{(2)}\|_{H^{m+1}(\Sigma)}\leq \delta$ with $\delta$ being defined in Theorem\,\ref{thm:LWP1} and $h^{(1)}$, $h^{(2)}\in H^{m+1}(\Sigma)$. Then by direct calculations, the associated solution $u(t,x,\varepsilon_1,\varepsilon_2)$ satisfies 
\begin{equation}\label{eq:u eps2}
\begin{cases}
\partial_t^2 u\left(t,x, \varepsilon_1, \varepsilon_2\right)-\Delta\left[c_1 u\left(x, t, \varepsilon_1, \varepsilon_2\right)+c_2 u^2\left(t,x, \varepsilon_1, \varepsilon_2\right)\right]=0 & \text { in } M, \\
u\left(0,x, \varepsilon_1, \varepsilon_2\right)=\partial_t u\left(0,x, \varepsilon_1, \varepsilon_2\right)=0 & \text { in } \Omega, \\
u\left(t,x, \varepsilon_1, \varepsilon_2\right)=\varepsilon_1 h^{(1)}+\varepsilon_2 h^{(2)} & \text { on } \Sigma .
\end{cases}
\end{equation}

Next, we let $v(t,x,\varepsilon_1,\varepsilon_2)=c_1u(t,x,\varepsilon_1,\varepsilon_2)$. Then it is directly verified that $v(t,x,\varepsilon_1,\varepsilon_2)$ satisfies
\begin{equation}\label{eq:v eps2}
\begin{cases}
\partial_t^2 v\left(t,x, \varepsilon_1, \varepsilon_2\right)-c_1\Delta\left[v\left(x, t, \varepsilon_1, \varepsilon_2\right)+\frac{c_2}{c_1^2} v^2\left(t,x, \varepsilon_1, \varepsilon_2\right)\right]=0 & \text { in } M, \\
v\left(0,x, \varepsilon_1, \varepsilon_2\right)=\partial_t v\left(0,x, \varepsilon_1, \varepsilon_2\right)=0 & \text { in } \Omega, \\
v\left(t,x, \varepsilon_1, \varepsilon_2\right)=c_1\left(\varepsilon_1 h^{(1)}+\varepsilon_2 h^{(2)}\right) & \text { on } \Sigma .
\end{cases}
\end{equation}
Write $v(t,x,\varepsilon_1,\varepsilon_2)$ as $v$ for short, and set
\begin{equation*}
v^{(12)}=\frac{\partial^2v}{\partial\varepsilon_1\partial\varepsilon_2}
\left|_{\varepsilon_1=\varepsilon_2=0}\right.,
\end{equation*}
which satisfies
\begin{equation}\label{eq:u12}
\left\{
\begin{aligned}
&v^{(12)}_{tt}-c_1\Delta v_1 =c_1\nabla \cdot G(v^{(1)},v^{(2)}) \ &&\text{in}\ M,\\
&v^{(12)}(0,x) = v^{(12)}_t(0,x) = 0                              \ &&\text{in}\ \Omega,\\
&v^{(12)}(t,x)= 0                                                 \ &&\text{on}\ \Sigma,
\end{aligned}
\right.
\end{equation}
where $G(v^{(1)},v^{(2)})=2\nabla (\frac{c_2}{c_1^2}v^{(1)}v^{(2)})$. For $j=1,2$, we set $v^{(j)}$ to be the solution of
\begin{equation}\label{eq:vj}
\left\{
\begin{aligned}
&v^{(j)}_{tt}-c_1\Delta v^{(j)}=0\ &&\text{in}\ M,\\
&v^{(j)}(0,x) = v^{(j)}_t(0,x) = 0\ &&\text{in}\ \Omega,\\
&v^{(j)}(t,x)=c_1 h^{(j)}(t,x)\ &&\text{on}\ \Sigma.
\end{aligned}
\right.
\end{equation}
In the same manner to (\ref{eq:kappa def}), the second order linearization part of the DtN map $\Lambda_{L,N}$ for (\ref{eq:u eps2}) can be computed and is given by
\begin{equation}\label{eq:kappa u12}
\Lambda^{(v,12)}_{L,N}:v\left|_{\Sigma}\right.~\mbox{or}~(h^{(1)},h^{(2)}) \mapsto  \frac{\partial^2 \Lambda_{L,N}}{\partial\varepsilon_1\partial\varepsilon_2}(\varepsilon_1h^{(1)}+\varepsilon_2h^{(2)})
\left|_{\varepsilon_1=\varepsilon_2=0}\right.=c_1\nu\cdot\left[\nabla v^{(12)}+G(v^{(1)},v^{(2)})\right].
\end{equation}
On the other hand, we consider the backward NPE equation
\begin{equation}\label{eq:dual w}
\left\{
\begin{aligned}
&w_{tt}-\nabla\cdot D^L(w)=0\ &&\text{in}\ M,\\
&w(T,x) = w_t(T,x) = 0\ &&\text{in}\ \Omega,\\
&w(t,x)= g(t,x)\ &&\text{on}\ \Sigma,
\end{aligned}
\right.
\end{equation}
where $g\in L^2(\Sigma)$. It follows from (\ref{eq:kappa u12}), (\ref{eq:u12}), integration by parts, and the fact that $v^{(12)}$ vanishes on the boundary that
\begin{align*}
&\int_0^T\int_{\partial\Omega}g\left[\Lambda^{(v,12)}_{L,N}(h^{(1)},h^{(2)})-c_1\nu\cdot G(v^{(1)},v^{(2)})\right]\mathrm{d}s\mathrm{d}t\\
=&\int_0^T\int_{\partial\Omega}c_1g\nu\cdot\nabla v^{(12)}\,\mathrm{d}s\mathrm{d}t\\
=&\int_0^T\int_{\Omega}\left[c_1w\Delta v^{(12)}+\nabla (c_1w)\cdot \nabla(v^{(12)})\right]\mathrm{d}x\mathrm{d}t\\
=&\int_0^T\int_{\Omega}\left[w(v^{(12)}_{tt}-c_1\nabla\cdot G(v^{(1)},v^{(2)}))+\nabla (c_1w)\cdot \nabla (v^{(12)})\right]\mathrm{d}x\mathrm{d}t\\
=&\int_0^T\int_{\Omega}\left[v^{(12)}w_{tt}+\nabla (c_1w)\cdot \nabla (v^{(12)})\right]\mathrm{d}x\mathrm{d}t-\int_{0}^T\int_{\Omega}c_1w\nabla\cdot G(v^{(1)},v^{(2)}))\, \mathrm{d}x\mathrm{d}t\\
=&\int_0^T\int_{\Omega}\left[v^{(12)}\Delta (c_1w)+\nabla (c_1w)\cdot \nabla (v^{(12)})\right]\mathrm{d}x\mathrm{d}t-\int_{0}^T\int_{\Omega} c_1w\nabla\cdot G(v^{(1)},v^{(2)}))\, \mathrm{d}x\mathrm{d}t\\
=&\int_0^T\int_{\partial \Omega}v^{(12)}\nu\cdot \nabla (c_1w)\mathrm{d}s\mathrm{d}t-\int_{0}^T\int_{\Omega} c_1w\nabla\cdot G(v^{(1)},v^{(2)}))\, \mathrm{d}x\mathrm{d}t\\
=&-\int_{0}^T\int_{\Omega} c_1w\nabla\cdot G(v^{(1)},v^{(2)}))\,\mathrm{d}x\mathrm{d}t.
\end{align*}
Therefore, we have
\begin{equation*}
\begin{aligned} &\int_0^T\int_{\partial\Omega}g\Lambda^{(v,12)}_{L,N}(h^{(1)},h^{(2)})\,\mathrm{d}s\mathrm{d}t\\
=&-\int_0^T\int_{\Omega}c_1w\nabla\cdot G(v^{(1)},v^{(2)})\,\mathrm{d}x\mathrm{d}t+\int_0^T\int_{\partial\Omega}c_1g\nu\cdot  G(v^{(1)},v^{(2)})\,\mathrm{d}s\mathrm{d}t\\
=&\int_0^T\int_{\Omega}\nabla (c_1w)\cdot G(v^{(1)},v^{(2)})\,\mathrm{d}x\mathrm{d}t,
\end{aligned}
\end{equation*}
which in turn implies
\begin{equation}\label{eq:c2}
2\int_0^T\int_{\Omega}\nabla (c_1w)\cdot \nabla \left(\frac{c_2}{c_1^2}v^{(1)}v^{(2)}\right)\,\mathrm{d}x\mathrm{d}t=\int_0^T\int_{\partial\Omega}g\Lambda^{(v,12)}_{L,N}(h^{(1)},h^{(2)})\,\mathrm{d}s\mathrm{d}t.
\end{equation}
Next, we shall recover the parameter $c_2$ from the integral (\ref{eq:c2}) by choosing Gaussian beam solutions for $w$ and $v^{(j)}$, $j=1,2$.
In particular, the solution $v$ will be constructed dependent on the linear coefficient $c_1$ which has been identified in the previous subsection.


\subsection{Gaussian beam solution}\label{sec:GB soln}
Consider the following equation
\begin{equation}\label{eq:u GB eqn}
\left\{
\begin{aligned}
&u_{tt}-c_1\Delta u=0\ &&\text{in}\ M,\\
&u(0,x) = u_{t}(0,x) = 0\ &&\text{in}\ \Omega,
\end{aligned}
\right.
\end{equation}
and assume its Gaussian beam solution $u_{\lambda}$ has the following form:
\begin{equation}\label{eq:u GB}
u_{\lambda}(t,x)=e^{\mathrm{i}\lambda \varphi(t,x)}a(t,x)+R_{\lambda}(t,x),
\end{equation}
where $\varphi$ is the phase function, $a$ is the amplitude function, and $R_{\lambda}$ is the residue function with $\lambda$ large enough.
Furthermore, we assume that $M=[0,T]\times \Omega$ is a Lorentzian manifold, and let $\Omega\Subset \widetilde{\Omega}$ with $c_1,\ c_2$ being smoothly extended onto $\widetilde{\Omega}$. Then $\widetilde{M}=[0,T] \times \widetilde{\Omega}$ is also a Lorentzian manifold.
	
Here, we would like to point out that $(\Omega,c_1^{-1}\mathrm{d}s^2)$ either has no conjugate points or satisfies the foliation condition, and hence every geodesic hits the boundary in a finite time (cf.\cite{PSU23,PSUZ19}).


\subsection{Fermi coordinates}\label{sec:Fermi coordinates}

To facilitate the construction of the Gaussian beam solutions, we next introduce the Fermi coordinates in order to calculate the phase function as well as the amplitude function in \eqref{eq:u GB}. Here, we shall establish a Fermi coordinate in the neighborhood of each null geodesic $\mathcal{V}$. Let $\mathcal{V}(t) = (t,\gamma(t))$, where $\gamma$ is the unit-speed geodesic in the Riemannian manifold $(\Omega, c_1^{-1}ds^2)$. Moreover, there exists a point $t_0\in (0,T)$ such that $\gamma(t_0)=x_0\in\Omega$ or $(x_0,t_0)\in M$. Based on {\bf Assumption I}, $\mathcal{V}$ can reach the boundary by moving forward or backward along the geodesic within a finite time (cf.\cite{PSU23,PSUZ19}), i.e. there exist $t_{-}$ and $t_{+}\in(0,T)$ such that $\gamma(t_{-})$ and $\gamma(t_{+})\in\partial \Omega$. For any $\varepsilon>0$, we extend the geodesic $\mathcal{V}$ to $\widetilde{M}=[0,T]\times\widetilde{\Omega}$, such that $\mathcal{V}$ is well-defined over the interval $[t_{-}-\varepsilon, t_{+}+\varepsilon]$.

Next, we introduce the Fermi coordinates following \cite{FIKO21,KLOU22,UW20}. Given $\gamma(t_0)=x_0\in\Omega$, let $T_{\gamma(t_0)}\Omega=T_{x_0}\Omega$ be the tangent space at $x_0$ and $\left\{\dot{\gamma}\left(t_0\right), \alpha_2, \alpha_3\right\}$ be the associated orthonormal basis.
Let $s$ be an arc length along $\gamma$ from $x_0$ where $s$ can be positive or negative, and $(\gamma(t_0+s),t_0+s)=\mathcal{V}(t_0+s)$. For $k=2,3$, let $e_{k}(s)\in T_{\gamma(t_0+s)}\Omega$ be the parallel transport of $\alpha_k$ along $\gamma$ to the point
$\gamma(t_0 + s)$, then $\left\{\dot{\gamma}\left(t_0+s\right), e_{2}(s), e_{3}(s)\right\}$ is the associated orthonormal basis at $\gamma(t_0 + s)$.
In particular, $e_{2}(0)=\alpha_2$ and $e_{3}(0)=\alpha_3$.

Define the coordinate $(y^1=s,y^2,y^3)$ through $\mathcal{F}_1:\mathbb{R}^3\to \widetilde{\Omega}$:
\begin{equation*}
\mathcal{F}_1(s,y^2,y^3)=\exp_{\gamma(t_0+s)}(y^2e_2(s)+y^3e_3(s)),
\end{equation*}
and	in the new coordinate, we have
\begin{equation*}
g\left|_{\gamma}\right.=\sum\limits_{j=1}^3(\mathrm{d}y^j)^2,\ \text{and}\ \frac{\partial g_{jk}}{\partial y^i}\left|_{\gamma}\right.=0,\ 1\leq i,j,k\leq 3,
\end{equation*}
Then the Euclidean metric $g_E$ takes the form
\begin{equation*}
	g_{E}=\sum\limits_{1\leq i,j\leq 3}c_1g_{ij}\mathrm{d}y^i\mathrm{d}y^j.
\end{equation*}

For any given $\varepsilon>0$ and $\delta>0$, the neighborhood (or tube) of the geodesic $\mathcal{V}$ is denoted by $V_{\varepsilon,\delta}$ as
\begin{equation*}
V_{\varepsilon,\delta}= \left\{(\tau,z')\in\widetilde{M};~ \tau\in \left[t_{-}-\frac{\varepsilon}{\sqrt{2}},t_{+}+\frac{\varepsilon}{\sqrt{2}},~|z'|<\delta\right]\right\}.
\end{equation*}
On the Lorentzian manifold $(\widetilde{M},-\mathrm{d}t^2+g)$ and along $V_{\varepsilon,\delta}$, the Fermi coordinates are defined as
\begin{equation*}
	z^0 = \tau = \frac{1}{\sqrt{2}}(t-t_0+s),\ z^1 = r = \frac{1}{\sqrt{2}}(-t+t_0+s),\ z^j=y^j,\ j=2,3.
\end{equation*}
Let $\tau_{\pm}=\sqrt{2}(t_{\pm}-t_0)$, then $\bar{g}=-\mathrm{d}t^2+g$ satisfies the following conditions
\begin{equation}
	\bar{g}\left|_{\mathcal{V}}\right.=2\mathrm{d}\tau \mathrm{d}r+\sum\limits_{j=2}^3(\mathrm{d}z^j)^2\ \text{and}\ \frac{\partial \bar{g}_{jk}}{\partial y^i}\left|_{\gamma}\right.=0,\ 1\leq i,j,k\leq 3.
\end{equation}
For short, $z =(\tau,z')=(\tau=z^0,r = z^1,z'')$ and $y = (s=y^1,y'')$, where $z'=(r,z^2,z^3)$, $z''=(z^2,z^3)$, and $y''=(y^2,y^3)$.
	

\subsection{Construction of Gaussion beams}\label{sec:Con GB}
Let the amplitude function of Gaussian beam solution \eqref{eq:u GB} has the following expansion:
\begin{equation}\label{eq:phi a expand}
a(\tau,z')=\chi\left(\frac{|z'|}{\delta}\right)\sum\limits_{k=0}^{N}\lambda^{-k}a_k(\tau,z'),
\end{equation}
where $N$ is a positive integer, $\delta$ is the small parameter, and $\chi:\mathbb{R}\to[0,+\infty)$ satisfies
\begin{equation}
\chi(v)=\left\{
\begin{aligned}
	&1,\ \text{when}\ |v|\leq \frac{1}{4},\\
	&0,\ \text{Otherwise}.
\end{aligned}
\right.
\end{equation}
Based on the exponential decay property of the Gaussian function, $\delta$ is small enough to guarantee that $a|_{t=0}=a|_{t=T}=0$.

For simplicity, we use the following notations for abbreviations
\begin{equation*}
\varphi_t = \frac{\partial \varphi}{\partial t},\ \varphi_{;j}=\frac{\partial\varphi}{\partial z_j},\
\varphi_{;jm}=\frac{\partial^2 \varphi}{\partial z_j\partial z_m},\ \varphi_{j;m}=\frac{\partial \varphi_j}{\partial z_m},\ j,m=1,2,3,  \text{and}\
\sum\limits_{j=1}^{3}\sum\limits_{m=1}^{3}a_{jm}b^{jm}=a_{jm}b^{jm}.
\end{equation*}
	
Next, we will rewrite the equation (\ref{eq:u GB eqn}) under the Fermi coordinate on $V_{\varepsilon,\delta}$. Substituting
(\ref{eq:u GB}) into (\ref{eq:u GB eqn}), we obtain
\begin{equation}\label{eq:u lambda tt}
	(u_{\lambda})_{tt} = e^{\mathrm{i}\lambda \varphi}[-\lambda^2a(\varphi_t)^2+\mathrm{i}\lambda(2a_t\varphi_t+a\varphi_{tt})+a_{tt}]+(R_{\lambda})_{tt}.
\end{equation}
By direct calculations, we have
\begin{align*}
(\nabla (u_{\lambda}))_{j}=(\nabla (e^{\mathrm{i}\lambda\varphi}a))_j=e^{\mathrm{i}\lambda\varphi}\left[\mathrm{i}\lambda a\varphi_{;j}+a_{;j}\right]+(\nabla(R_{\lambda}))_{j},
\end{align*}
and
\begin{align*}
&(\nabla (u_{\lambda}))_{j;m}=\left\{e^{\mathrm{i}\lambda\varphi}\left[\mathrm{i}\lambda a\varphi_{;j}+a_{;j}\right]\right\}_{;m}+(\nabla R_{\lambda})_{j;m}\\
=&e^{\mathrm{i}\lambda\varphi}\left\{-\lambda^2a\varphi_{;j}\varphi_{;m}\right.+\mathrm{i}\lambda\left[a_{;j}\varphi_{;m}+a_{;m}\varphi_{;j}+a\varphi_{;jm}\right]+a_{;jm}\left.\right\}+(\nabla R_{\lambda})_{j;m},
\end{align*}
which implies
\begin{equation}\label{eq:DL u}
\begin{aligned}
c_1\Delta (u_{\lambda}) &= c_1(\nabla (u_{\lambda}))_{j;m}g^{jm}c_1^{-1}\\
&=-\lambda^2e^{\mathrm{i}\lambda \varphi}(a\varphi_{;j}\varphi_{;m})g^{jm}\\
&+\mathrm{i}\lambda e^{\mathrm{i}\lambda\varphi}(a_{;j}\varphi_{;m}+a_{;m}\varphi_{;j}+a\varphi_{;jm})+e^{\mathrm{i}\lambda\varphi}(a_{;jm})g^{jm}+\nabla (R_{\lambda})_{j;m}g^{jm},
\end{aligned}
\end{equation}
where $\{g^{jm}\}$ is the inverse matrix of $\{g_{jm}\}$. Hence, the residue function $R_{\lambda}$ satisfies
\begin{equation}\label{eq:R lambda}
\left\{
\begin{aligned}
&(R_{\lambda})_{tt}-c_1\Delta R_{\lambda}=F_{\lambda}\ &&\text{in}\ M,\\
&R_{\lambda}(0,x) = (R_{\lambda})_{t}(0,x)=0\ &&\text{in}\ \Omega,\\
&R_{\lambda}=0\  && \text{on}\ \Sigma,\\
\end{aligned}
\right.
\end{equation}
where $F_{\lambda}=-(u_{\lambda})_{tt}+c_1 \Delta u_{\lambda}$.

Substituting (\ref{eq:phi a expand}), (\ref{eq:u lambda tt}), and (\ref{eq:DL u})
into (\ref{eq:R lambda}), it yields
\begin{equation*}
(u_{\lambda})_{tt}-c_1 \Delta u_{\lambda}=e^{\mathrm{i}\lambda\varphi}\left(\lambda^2 I_1+\sum\limits_{k=0}^N\lambda^{1-k} I_{k+2}+\mathcal{O}(\lambda^{-N})\right),
\end{equation*}
where $\{I_{k}\}_{1}^{N+2}$ are coefficients to be determined. Moreover, the following expansions hold by direct computations
\begin{equation*}
(u_{\lambda})_{tt} = e^{\mathrm{i}\lambda \varphi}\left[-\lambda^2\varphi_t^2a_0+\lambda\left(2\mathrm{i}\varphi_t(a_{0})_t+\mathrm{i}a_0\varphi_{tt}-\varphi_t^2a_1\right)\right]+\mathcal{O}(1),
\end{equation*}
and
\begin{equation*}
\begin{aligned}
c_1\Delta u_{\lambda}=\lambda^2e^{\mathrm{i}\lambda\varphi}(-a_0\varphi_{;j}\varphi_{;m})g^{jm}+\lambda e^{\mathrm{i}\lambda \varphi}\left[\mathrm{i}\left(a_{0;j}\varphi_{;m}+a_{0;m}\varphi_{;j}+a_0\varphi_{;jm}\right)-a_1\varphi_{;j}\varphi_{;m}\right]g^{jm}+\mathcal{O}(1).
\end{aligned}
\end{equation*}
Comparing the coefficients, we have
\begin{equation*}
I_1 = a_0\varphi_{;j}\varphi_{;m}g^{jm}-a_0(\varphi_t)^2,
\end{equation*}
and
\begin{align}
\label{eq:I2}
I_2 =& a_1\varphi_{;j}\varphi_{;m}g^{jm}-a_1(\varphi_t)^2+\mathrm{i}(2a_{0;t}\varphi_t+a_0\varphi_{tt})-\mathrm{i}(a_{0;m}\varphi_{;j}+a_0\varphi_{;jm}+\varphi_{;m}a_{0;j})g^{jm}.
\end{align}
Following \cite{UZ21b}, the coefficients $\{I_{k}\}_{1}^{N+2}$ should satisfy the following requirements
\begin{equation}\label{eq:cond I}
\frac{\partial^{|\alpha|}}{\partial z^{\alpha}}I_k=0\ \text{on}\ \mathcal{V}, \quad 1\leq k\leq N+2,
\end{equation}
where $\alpha = (0,\alpha_1,\alpha_2,\alpha_3)$, $\alpha_j \geq 0$, $j=1,2,3$, and $|\alpha|=\sum\limits_{j=1}^{3}\alpha_j\leq N$.
Therefore, $R_{\lambda}$ has the following estimate (cf.\cite{FIKO21,UZ21a})
\begin{equation*}
\|R_{\lambda}\|_{W^{1,3}(M)}=\mathcal{O}(\lambda^{-1/2}).
\end{equation*}
	

\subsection{Construction of the phase function}\label{sec:phase function}

We now construct phase function $\varphi$. Let
\begin{equation}\label{eq:def S}
	S\varphi:=-(\partial_t\varphi)^2+c_1|\nabla\varphi|^2.
\end{equation}	
Consider the point $\gamma(t_0)$, given the associated Fermi coordinates $z=(\tau,z^1,z^2,z^3)$, the phase function satisfies
\begin{equation}\label{eq:cond S}
	\frac{\partial^{|\alpha|}}{\partial z^{\alpha}}(S\varphi)(\tau,0,0,0)=0,
\end{equation}	
for any $\alpha=(0,\alpha_1,\alpha_2,\alpha_3)$ with $|\alpha|\leq N$.

Let $\tilde{g} = -\mathrm{d}t^2+g_{c_1}$, where $g_{c_1}$ is given by (\ref{eq:gc1}), then (\ref{eq:def S}) and (\ref{eq:cond S}) can be rewritten as
\begin{eqnarray*}
S\varphi = \langle \mathrm{d}\varphi,\mathrm{d}\varphi\rangle_{\tilde{g}} \quad \mbox{and} \quad
\frac{\partial^{|\alpha|}}{\partial z^{\alpha}}\langle \mathrm{d}\varphi,\mathrm{d}\varphi\rangle_{\tilde{g}} =0~~ \text{on}~~V_{\epsilon,\delta},
\end{eqnarray*}
where $\langle \cdot,\cdot\rangle_{\tilde{g}}$ is the inner production on $\mathcal{V}$.

Suppose $\varphi=\sum\limits_{k=0}^N\varphi_k(\tau,z')$ with $\varphi_k$ being complex valued homogeneous polynomials of degree $k$,
and the first three terms are given by
\begin{eqnarray*}
	\varphi_0=0,\ \varphi_1 = \frac{-t+t_0+s}{\sqrt{2}},\ \varphi_2 = \sum\limits_{1\leq i,j\leq 3}H_{ij}(\tau)z^iz^j,
\end{eqnarray*}
where $H$ is symmetry, $\Im H(\tau)>0$, and satisfies the following Ricatti ODE:
\begin{eqnarray}\label{eq:Ricatti ODE}
\frac{\mathrm{d}}{\mathrm{d}\tau}H+HCH+D=0,\quad \tau\in(\tau_{-}-\frac{\varepsilon}{2},
\tau_{+}+\frac{\varepsilon}{2}), \quad H(0)=H_0.
\end{eqnarray}
Here, $C$ and $D$ are given by
\begin{align}\label{eq:def CD}
C=	\begin{pmatrix}
0 & 0 & 0 \\
0 & 2 & 0 \\
0 & 0 & 2
\end{pmatrix},\
D=\frac{1}{4}\begin{pmatrix}
\frac{\partial^2}{\partial_{1}^2} & \frac{\partial^2}{\partial_{1}\partial_{2}} & \frac{\partial^2}{\partial_{1}\partial_{3}} \\
\frac{\partial^2}{\partial_{2}\partial_{1}} & \frac{\partial^2}{\partial_{2}^2} & \frac{\partial^2}{\partial_{2}\partial_{3}} \\
\frac{\partial^2}{\partial_{3}\partial_{1}} & \frac{\partial^2}{\partial_{3}\partial_{2}} & \frac{\partial^2}{\partial_{3}^2}
\end{pmatrix}g^{11},
\end{align}
where $g^{11}$ is the first entry of the inverse matrix $\{g^{jm}\}$. 	

The following lemma shows the existence of a solution to the Ricatti equation (\ref{eq:Ricatti ODE}).
\begin{lemma}\label{lem:existence}(cf.\cite{FIKO21})
Let $Y(\tau)$ and $Z(\tau)$ solve the ODEs
\begin{align*}
	&\frac{\mathrm{d}}{\mathrm{d}\tau}Y(\tau)=CZ(\tau),\quad Y(0)=Y_0,\\
    &\frac{\mathrm{d}}{\mathrm{d}\tau}Z(\tau)=-D(\tau)Y(\tau),\quad Z(0)=Y_1 = H(0)Y_0,
\end{align*}
respectively. Then the Ricatti equation (\ref{eq:Ricatti ODE}) has a unique solution $H(\tau) = Z(\tau)(Y(\tau))^{-1}$.
Moreover, $H$ is symmetric and $\Im H(\tau)>0$ for all $\tau\in(\tau_{-}-\frac{\varepsilon}{2}, \tau_{+}+\frac{\varepsilon}{2})$.

In addition, $Y(\tau)$ is non-degenerate and
\begin{equation*}
		det(\Im H(\tau))|det(Y(\tau))|^2=C,
\end{equation*}
where $C$ is a constant independent of $\tau$.
\end{lemma}

\begin{remark}\label{remark:SI}
The construction of $S$ depends on the properties of $I_1$ under the requirement (\ref{eq:cond I}).

When $|\alpha|=0$ in (\ref{eq:cond I}) with $a_0\neq 0$, $I_1=a_0S\varphi=0$ on $\mathcal{V}$.

When $|\alpha|>0$, for example, $\alpha=(1,0,0)$, the requirement (\ref{eq:cond I}) implies
\begin{equation*}
\left[ \varphi_{;j}\varphi_{;m}g^{jm}-(\varphi_t)^2\right]\frac{\partial a_0}{\partial z^1}+a_0\frac{\partial}{\partial z^1} \left[ \varphi_{;j}\varphi_{;m}g^{jm}-(\varphi_t)^2\right]=0,
\end{equation*}
which is equivalent to
\begin{equation*}
		S\varphi(\tau,0)\frac{\partial a_0}{\partial z^1}+a_0\frac{\partial S\varphi}{\partial z^1}(\tau,0)=0
\end{equation*}
under the Fermi coordinates on $\mathcal{V}$.

\end{remark}


\subsection{Construction of the amplitude function}\label{sec:amplitude function}

Here, we calculate the amplitude function $a_0$. Suppose $a_0$ has the following form
\begin{equation*}
a_0(\tau,z^1,z^2,z^3) = A(\tau,z^1,z^2,z^3)\varphi_{;1},
\end{equation*}
where $A(\tau,z^1,z^2,z^3)=A(\tau,z^1,z^2,z^3)\neq 0$ is to be determined. According to the definitions in Section\,\ref{sec:Fermi coordinates},
the above equation can be rewritten as
\begin{equation*}
a_0=a_0(t,s,y^1,y^2)=A(t,s,y^1,y^2)\varphi_{;1},
\end{equation*}
where the phase function $\varphi$ is computed in Section\,\ref{sec:phase function}.

If we consider the expression of $\varphi$ as
\begin{equation*}
\varphi(\tau,z^1,z^2,z^3)=\varphi_0(\tau,z^1,z^2,z^3)+\varphi_1(\tau,z^1,z^2,z^3)=z^1=\frac{-t+t_0+s}{\sqrt{2}},
\end{equation*}
then on the geodesic $\mathcal{V}$, we know
\begin{equation*}
\varphi_{;1}=\frac{1}{\sqrt{2}},\quad  \varphi_{;j}=0,\ j=2,3,\quad  \partial_{\tau}\varphi=0,\quad \varphi_t=-\frac{1}{\sqrt{2}},\
\end{equation*}
which implies
\begin{equation*}
a_0\left|_{\mathcal{V}}\right. = \frac{A(\tau,z^1,z^2,z^3)}{\sqrt{2}}.
\end{equation*}
It follows from the Fermi coordinates in Section\,\ref{sec:Fermi coordinates} that
\begin{align*}
(a_{0})_{;t}=\frac{1}{\sqrt{2}}A_t+A\frac{\partial^2\varphi}{\partial s\partial t}, \quad
\varphi_{;11}=\frac{\partial^2\varphi}{\partial s^2}+\frac{1}{\sqrt{2c_0}}\frac{\partial c_0}{\partial s}, \quad
\varphi_{;jj}=\frac{\partial^2\varphi}{\partial y^j\partial y^j},~j=2,3,
\end{align*}
where $c_0=\sqrt{c_1}$.

For any $|\alpha| \leq N$, the phase function $\varphi$ satisfies (\ref{eq:cond I}). In case of $k=2$ and $|\alpha|=0$, by direct calculations,
the equation (\ref{eq:I2}) can be rewritten as
\begin{align}\label{eq:A eqn}
\frac{1}{\sqrt{2}}\left(\varphi_{tt}-\varphi_{ss}-\sum\limits_{j=2}^3\frac{\partial^2\varphi}{\partial y^{j}\partial y^{j}}\right)A-\sqrt{2}\left(\frac{\partial^2\varphi}{\partial s^2}+\frac{\partial^2 \varphi}{\partial s\partial t}\right)A-(A_t+A_s)-\frac{3}{2c_1}\frac{\partial c_1}{\partial s}A=0
\end{align}

For the first term in the last equation, it follows from Lemma\,\ref{lem:existence} that
\begin{align*}
\left(\varphi_{tt}-\varphi_{ss}-\sum\limits_{j=2}^3\frac{\partial^2\varphi}{\partial y^{j}\partial y^{j}}\right)
=-\sum\limits_{j=2}^3\frac{\partial^2\varphi}{\partial y^{j}\partial y^{j}}
=-Tr(CH) = -\frac{\partial}{\partial \tau}\log(\det(Y(\tau))),
\end{align*}
where $C$ is given by (\ref{eq:def CD}).

For the second term in (\ref{eq:A eqn}), we have
\begin{align*}
&-\sqrt{2}\left(\frac{\partial^2\varphi}{\partial s^2}+\frac{\partial^2 \varphi}{\partial s\partial t}\right)\
=-2\frac{\partial^2 \varphi}{\partial s\partial \tau}=0.
\end{align*}

Since $\tau = \frac{t-t_0+s}{\sqrt{2}}$, it yields
\begin{align*}
\frac{\partial A}{\partial\tau}=\frac{1}{\sqrt{2}}\left(\frac{\partial A}{\partial t}+\frac{\partial A}{\partial s}\right),
\end{align*}
which implies
\begin{align*}
A_t+A_s = \sqrt{2}\partial_{\tau}A.
\end{align*}

Combining the above equalities, the equation (\ref{eq:A eqn}) can be simplified as
\begin{align*}
2\frac{\partial}{\partial\tau}A+\left(\frac{\partial}{\partial\tau}\log(\det{Y(\tau)})+\frac{3}{\sqrt{2}c_1}\frac{\partial}{\partial\tau} c_1\right)A=0,
\end{align*}
which gives
\begin{align*}
\frac{\partial }{\partial\tau}\left[\log (A^2)+\log(\det(Y(\tau)))+\frac{3}{\sqrt{2}}\log(c_1) \right]=0.
\end{align*}
Hence
\begin{align*}
A=\widetilde{C} \det(Y(\tau))^{-\frac{1}{2}}c_1^{-\frac{3}{2\sqrt{2}}},
\end{align*}
where $\widetilde{C}$ is the constant independent of $\tau$ and $Y$ is defined by Lemma\,\ref{lem:existence}.


\subsection{Determination of the non-linear coefficient $c_2$}\label{sec:c2}

In this subsection, we prove the unique identifiability of $c_2$ from the mapping $\Lambda_{L,N}$ based on the stationary phase lemma.

For any point $x_0\in \Omega$, we let $p=(\frac{T}{2},x_0)\in M$, $\xi^{(k)}=(\xi^{(k)}_1,\xi^{(k)}_2)^{T}\in T^{*}_{x_0}\Omega$ with $|\xi^{(k)}|=1$, $k=0,1,2$, and then define
\begin{equation*}
L_p^{*}M = \{(\tau,\xi)\in T^{*}_pM,\tau^2 = c_1|\xi|^2\},
\end{equation*}
and choose linearly dependent vectors $\{\zeta^{(k)}\}_{k=0}^{2}\subset L_p^{*}M$ by $\zeta^{(k)}=(\sqrt{c_1},\xi^{(k)})$, $k=0,1,2$.
	
\begin{remark}
We would like to remark that the definition of $L_p^{*}M$ involves the coefficient $c_1$, which has been uniquely determined in Section~\ref{sec:first-order}. Hence, this definition is well-defined. This definition can be understood as selecting a curve $\tau^2 = c_1|\xi|^2$ in the tangential space of $M$.
\end{remark}

Since $\zeta^{(0)}$, $\zeta^{(1)}$, and $\zeta^{(2)}$ are the linearly dependent, there are non-zero constant $\kappa_0$, $\kappa_1$, and $\kappa_2$ such that
\begin{equation*}
	\kappa_0\zeta^{(0)}+\kappa_1\zeta^{(1)}+\kappa_2\zeta^{(2)}=0,
\end{equation*}
which is equivalent to
\begin{equation*}
	\kappa_0\xi^{(0)}+\kappa_1\xi^{(1)}+\kappa_2\xi^{(2)}=0.
\end{equation*}
	
Let $\mathcal{V}^{(k)}$ to be the null geodesics in the Lorentzian manifold $(M,-dt^2+g)$ with the cotangent vector $\zeta^{(k)}$ at the point $p$.
Then we shall construct the Gaussian beam solution $u_{\lambda}^{(k)}$ to be concentrated near the null geodesics $\mathcal{V}^{(k)}$.
For $k=0, 1, 2$, we let the Gaussian beam solution $u_{\lambda}^{(k)}$ be
\begin{equation}\label{eq:GB}
	u_{\lambda}^{(k)}=\chi^{(k)}e^{\mathrm{i}\lambda \kappa_k \varphi^{(k)}}(a^{(k)}+\mathcal{O}(\lambda^{-1})),
\end{equation}
where the phase function $\varphi^{(k)}$ satisfies
\begin{equation}\label{eq:varphi xi}
	\nabla \varphi^{(k)}(p)=\xi^{(k)},\ \text{and}\ \varphi^{(k)}\ \text{vanishes along}\ \mathcal{V}^{(k)},
\end{equation}
and the amplitude function $a^{(k)}$ is
\begin{equation}\label{eq:a xi}
	a^{(k)}(p) = \xi^{(k)}_1\neq 0.
\end{equation}

\begin{remark}
Since $dim(T^{*}_{x_0}\Omega)=2$, we can choose $\{\xi^{(k)}\}_{k=0}^{2}$ arbitrarily to guarantee $\xi^{(k)}_1\neq 0$, $k=0, 1, 2$.
\end{remark}

Let
\begin{equation*}
w=u^{(0)}_{\lambda}, \quad v^{(1)}=u_{\lambda}^{(1)}, \quad v^{(2)}=u^{(2)}_{\lambda},
\end{equation*}
in the identity (\ref{eq:c2}). Then on the boundary $\Sigma$, we have
\begin{equation*}
g = u_{\lambda}^{(0)}, \quad h^{(1)} = (u_{\lambda}^{(1)}/c_1), \quad h^{(2)} = (u_{\lambda}^{(2)}/c_1).
\end{equation*}
Furthermore, for $k=0, 1, 2$, we have $u_{\lambda}^{(k)}\left|_{t=0}\right.=\partial_tu_{\lambda}^{(k)}\left|_{t=0}\right.=0$. Together with (\ref{eq:GB}), the identity (\ref{eq:c2}) can be rewritten as
\begin{equation}
\begin{split}
&\int_0^T\int_{\partial\Omega}g\Lambda^{(u,12)}_{L,N}(h^{(1)},h^{(2)})\,\mathrm{d}s\mathrm{d}t\\
=&2\int_0^T\int_{\Omega}\nabla (c_1w)\cdot \nabla \left(\frac{c_2}{c_1^2}v^{(1)}v^{(2)}\right)\,\mathrm{d}x\mathrm{d}t\\
=&-2\lambda^2\int_{0}^T\int_{\Omega}\chi^{(0)}\chi^{(1)}\chi^{(2)}\left[\kappa_0\nabla\varphi^{(0)}\cdot(\kappa_1\nabla\varphi^{(1)}+\kappa_2\nabla \varphi^{(2)})\right]\\
&\hspace*{2cm}\times c_2a^{(0)}a^{(1)}a^{(2)}e^{\mathrm{i}\lambda(\kappa_0\varphi^{(0)}+\kappa_1\varphi^{(1)}+\kappa_2\varphi^{(2)})}\,\mathrm{d}x\mathrm{d}t\\
&+\mathcal{O}(\lambda),
\end{split}
\end{equation}
which in turn implies that
\begin{align}\label{eq:c2 eqn}
&\lambda^{-2}\int_0^T\int_{\partial\Omega}g\Lambda^{(u,12)}_{L,N}(h^{(1)},h^{(2)})\mathrm{d}s\mathrm{d}t =-2\int_{0}^T\int_{\Omega}e^{\mathrm{i}\lambda \widetilde{\varphi}}\widetilde{a} \mathrm{d}x\mathrm{d}t+\mathcal{O}(\lambda^{-1}),
\end{align}
where
\begin{align*}
&\widetilde{\varphi} = \kappa_0\varphi^{(0)}+\kappa_1\varphi^{(1)}+\kappa_2\varphi^{(2)},\\
&\widetilde{a}=\chi^{(0)}\chi^{(1)}\chi^{(2)}\left[\kappa_0\nabla\varphi^{(0)}\cdot(\kappa_1\nabla\varphi^{(1)}+\kappa_2\nabla \varphi^{(2)})\right]c_2a^{(0)}a^{(1)}a^{(2)}.
\end{align*}

We proceed with the recovery the non-linear coefficient $c_2$. To that end, we need to make use of the stationary phase lemma, which is introduced as follows. 

\begin{lemma}\label{lem:sp lem}(cf.\cite{Ma20}) Consider the oscillatory integral $I(\lambda)$ as
\begin{equation*}
I(\lambda) = \int_{\mathbb{R}^n}e^{\mathrm{i}\lambda\varphi(p)}a(p)dp.
\end{equation*}
For an arbitrary integer $N\in \mathbb{N}$, we assume that
\begin{itemize}
	\item[(1)] $a\in C^{n+2N+3}(\mathbb{R}^n;\mathbb{R})$ with $\sum_{|\alpha|<n+2N+3}\sup_{\mathbb{R}^n}|\partial^{\alpha}a|<+\infty$.
	\item[(2)] $\varphi\in C^{n+2N+6}(\mathbb{R}^n;\mathbb{C})$ with $\sum_{|\alpha|<n+2N+6}\sup_{\mathbb{R}^n}|\partial^{\alpha}\varphi|<+\infty$.
	\item[(3)] $p_0$ is the only critical point of $\varphi(p)$ on supp$(a(p))$ $\mathit{i.e.}$ $\varphi(p_0)=\nabla \varphi(p_0)=0, \nabla \varphi(p)\neq 0$ for $p\neq p_0$.
	\item[(4)] The Hessian $D^2 [\varphi]:=\{\frac{\partial^2\varphi}{\partial {y^j}\partial {y^k}}\}_{j,k=1}^n$ satisfies $\det(D^2 [\varphi])(p_0)\neq 0$.
\end{itemize}
Then $I(\lambda)$ is well-defined in the oscillatory integral sense, and as $\lambda\to+\infty$, there is
\[
I(\lambda) = \left(\frac{2\pi}{\lambda}\right)^{n/2}\frac{e^{\mathrm{i}\lambda \varphi(p_0)+\mathrm{i}\frac{\pi}{4}\text{sgn}(D^2[\varphi](p_0))}}{|\det(D^2[\varphi](p_0))|^{1/2}}a(p_0)+(\lambda^{-\frac{n}{2}-N-1}).
\]
\end{lemma}

The following lemma shows that $\widetilde{\varphi}$ and $\widetilde{a}$ satisfy the conditions in Lemma\,\ref{lem:sp lem}.
\begin{lemma}(cf.\cite{FIKO21}) For any point $x_0\in \Omega$, let $p=(\frac{T}{2},x_0)\in M$, then $\widetilde{\varphi}$ in (\ref{eq:c2 eqn}) satisfies
\begin{itemize}
\item[(1)]$\widetilde{\varphi}(p)=0$,
\item[(2)]$\nabla \widetilde{\varphi}(p)=0$,
\item[(3)]$\Im (\widetilde{\varphi}(q))\ge C_1 (d(p,q))^2$ for $q$ in a neighbourhood of $p$ with constant $C_1>0$,
\item[(4)]$D^2\Im (\widetilde{\varphi}(X,X))>0$, for all $X\in L_p^{*}M\setminus \{0\}$.
\end{itemize}
\end{lemma}

According to Lemma~\ref{lem:sp lem}, the oscillatory integral in the right hand side of (\ref{eq:c2 eqn}) mainly depends on
\begin{equation*}
	\widetilde{a}(p)=\left[\chi^{(0)}\chi^{(1)}\chi^{(2)}\left[\kappa_0\nabla\varphi^{(0)}\cdot(\kappa_1\nabla\varphi^{(1)}+\kappa_2\nabla \varphi^{(2)})\right]c_2a^{(0)}a^{(1)}a^{(2)}\right](p),
\end{equation*}
as $\lambda\to+\infty$. It follows from the arbitrariness of $\{\xi^{(k)}\}_{k=0}^{2}$, as well (\ref{eq:varphi xi}), and (\ref{eq:a xi}) that
\begin{equation*}
	\kappa_0\xi^{(0)}\cdot(\kappa_1\xi^{(1)}+\kappa_2\xi^{(2)}) = -\|\kappa_0\xi^{(0)}\|^2\neq 0,
\end{equation*}
which readily yields the uniqueness of determining $c_2$.

In the rest of the section, we extend all the previous results to the case $n>2$ for \eqref{eq:NPE order2}. We mainly present the necessary modifications.  
Consider the following boundary data
\begin{equation*}
	h(t,x) = \varepsilon_1 h^{(1)}+\varepsilon_2 h^{(2)}+\cdots+\varepsilon_nh^{(n)},
\end{equation*}
where $\sum_{j=1}^n\varepsilon_j\|h^{(j)}\|_{H^{m+1}(\Sigma)}\leq \delta$ with $\delta$ being defined in Theorem~\ref{thm:LWP1} and $h^{(j)}\in H^{m+1}(\Sigma)$.
Then the associated solution $u(x,t,\varepsilon_1,\cdots,\varepsilon_n)$ satisfies 
\begin{equation*}\label{eq:u epsn}
\left\{
\begin{aligned}
	&\partial_t^2u(x,t,\varepsilon_1,\cdots,\varepsilon_n)+\Delta \left[c_1u(x,t,\varepsilon_1,\cdots,\varepsilon_n)+c_nu^n(x,t,\varepsilon_1,\cdots,\varepsilon_n)\right]=0\ &\text{in}\ M,\\
	&u(x,0,\varepsilon_1,\cdots,\varepsilon_n)=\partial_tu(x,0,\varepsilon_1,\cdots,\varepsilon_n) = 0\ &\text{in}\ \Omega.\\
	&u(x,t,\varepsilon_1,\cdots,\varepsilon_n)= \varepsilon_1h^{(1)}+\cdots+\varepsilon_nh^{(n)}\ &\text{on}\ \Sigma.
\end{aligned}
\right.
\end{equation*}
Write $u(x,t,\varepsilon_1,\cdots,\varepsilon_n)$ as $u$ for short. Define the linearization part of $u$ as
\begin{align*}
u^{(j)}=\frac{\partial u}{\partial{\varepsilon_j}}\left|_{\varepsilon_1=\cdots=\varepsilon_n=0}\right.,~\mbox{for}~j=1,\cdots,n,
\end{align*}
which satisfies
\begin{equation*}
\left\{
\begin{aligned}
&u^{(j)}_{tt} -\Delta (c_1u^{(j)})=0\ &&\text{in}\ M,\\
&u^{(j)}(0,x) = u^{(j)}_t(0,x) = 0\ &&\text{in}\ \Omega,\\
&u^{(j)}(t,x)= h^{(j)}(t,x)\ &&\text{on}\ \Sigma.
\end{aligned}
\right.
\end{equation*}
Then we can recover the linear coefficient $c_1$ uniquely by following a similar argument to that in Section~\ref{sec:first-order}.

Let $v(t,x,\varepsilon_1,\cdots,\varepsilon_n)=c_1u(t,x,\varepsilon_1,\cdots,\varepsilon_n$, then $v(t,x,\varepsilon_1,\cdots,\varepsilon_n)$ satisfy
\begin{equation*}
\begin{cases}
\partial_t^2 v\left(t,x,\varepsilon_1,\cdots,\varepsilon_n\right)-c_1\Delta\left[v\left(x, t,\varepsilon_1,\cdots,\varepsilon_n\right)+\frac{c_2}{c_1^n} v^n\left(t,x,\varepsilon_1,\cdots,\varepsilon_n\right)\right]=0 & \text { in } M, \\
v\left(0,x,\varepsilon_1,\cdots,\varepsilon_n\right)=\partial_t v\left(0,x,\varepsilon_1,\cdots,\varepsilon_n\right)=0 & \text { in } \Omega, \\
v\left(t,x,\varepsilon_1,\cdots,\varepsilon_n\right)=c_1\left(\varepsilon_1 h^{(1)}+\cdots+\varepsilon_n h^{(n)}\right) & \text { on } \Sigma .
\end{cases}
\end{equation*}
Write $v(t,x,\varepsilon_1,\cdots,\varepsilon_n)$ as $v$ for short, and set
\begin{align*}
v^{(1\cdots n)} = \frac{\partial^n v}{\partial\varepsilon_1\cdots\partial\varepsilon_n}\left|_{\varepsilon_1=\cdots=\varepsilon_n=0}\right.,
\end{align*}
which satisfies
\begin{equation*}\label{eq:u12}
\left\{
\begin{aligned}
&v^{(1\cdots n)}_{tt}-c_1\Delta v^{(1\cdots n)} =c_1\nabla \cdot G(v^{(1)},\cdots,v^{(n)}) \ &&\text{in}\ M,\\
&v^{(1\cdots n)}(0,x) = v^{(1\cdots n)}_t(0,x) = 0                              \ &&\text{in}\ \Omega,\\
&v^{(1\cdots n)}(t,x)= 0                                                 \ &&\text{on}\ \Sigma,
\end{aligned}
\right.
\end{equation*}
where $G(v^{(1)},\cdots,v^{(n)})=n\nabla (\frac{c_2}{c_1^n}v^{(1)} \cdots v^{(n)})$. The associated DtN map is defined by
\begin{equation}\label{eq:kappa u12}
\Lambda^{(v,1\cdots n)}_{L,N}:v\left|_{\Sigma}\right.~\mbox{or}~(h^{(1)},\cdots,h^{(n)}) \mapsto c_1\nu\cdot\left[\nabla v^{(1\cdots n)}+G(v^{(1)},\cdots,v^{(n)})\right].
\end{equation}
The integral identity for recovering the parameter $c_n$ corresponding to (\ref{eq:c2}) is given by
\begin{equation}\label{eq:cn}
n\int_0^T\int_{\Omega}\nabla (c_1w)\cdot \nabla \left(\frac{c_2}{c_1^n}v^{(1)}\cdots v^{(n)}\right)\mathrm{d}x\mathrm{d}t
=\int_0^T\int_{\partial\Omega}g\Lambda^{(v,1\cdots n)}_{L,N}(h^{(1)},\cdots,h^{(n)})\,\mathrm{d}s\mathrm{d}t,
\end{equation}
where $w$ and $g$ are given in the dual problem (\ref{eq:dual w}).  Here, we only show the construction of the nonlinear coefficient $c_n$.

For any point $x_0\in \Omega$, we let $p=(\frac{T}{2},x_0)\in M$, $\xi^{(k)}=(\xi^{(k)}_1,\cdots,\xi^{(k)}_d)^{T}\in T^{*}_{x_0}\Omega$ with $|\xi^{(k)}|=1$, $k=0,1,\cdots,n$, and then define
\begin{equation*}
L_p^{*}M = \{(\tau,\xi)\in T^{*}_pM,\tau^2 = c_1|\xi|^2\},
\end{equation*}
and choose linearly dependent vectors $\{\zeta^{(k)}\}_{k=0}^{n}\subset L_p^{*}M$ by $\zeta^{(k)}=(\sqrt{c_1},\xi^{(k)})$, $k=0,1,\cdots,n$.
	
Since $\{\zeta^{(k)}\}_{k=0}^{n}$ are linearly dependent, there are non-zero constants $\kappa_0$, $\kappa_1$, $\cdots$, $\kappa_n$ such that
\begin{equation*}
	\sum\limits_{k=0}^{n}\kappa_k\zeta^{(k)}=0 \quad \mbox{and} \quad \sum\limits_{k=0}^{n}\kappa_k\xi^{(k)}=0.
\end{equation*}
	
Let $\mathcal{V}^{(k)}$, $k=0,1,\cdots,n$ be the null geodesics in the Lorentzian manifold $(M,-dt^2+g)$ with the cotangent vector $\zeta^{(k)}$ at the point $p$. Next we construct the Gaussian beam solution $u_{\lambda}^{(k)}$ to be concentrated near the null geodesics $\mathcal{V}^{(k)}$.
For $k=0,1,\cdots,n$, we set the Gaussian beam solution $u_{\lambda}^{(k)}$ to be of the form
\begin{equation}\label{eq:GBn}
	u_{\lambda}^{(k)}=\chi^{(k)}e^{\mathrm{i}\lambda \kappa_k \varphi^{(k)}}(a^{(k)}+\mathcal{O}(\lambda^{-1})),
\end{equation}
where the phase function $\varphi^{(k)}$ satisfies
\begin{equation}\label{eq:varphi xin}
	\nabla \varphi^{(k)}(p)=\xi^{(k)},\ \text{and}\ \varphi^{(k)}\ \text{vanishes along}\ \mathcal{V}^{(k)},
\end{equation}
and the amplitude function $a^{(k)}$ is
\begin{equation}\label{eq:a xin}
	a^{(k)}(p) = \xi^{(k)}_1\neq 0.
\end{equation}

Let
\begin{equation*}
w=u^{(0)}_{\lambda}, \quad v^{(k)}=u_{\lambda}^{(k)},\quad k=0,1,\cdots,n,
\end{equation*}
in the identity (\ref{eq:cn}). Then on the boundary $\Sigma$, we have
\begin{equation*}
g = u_{\lambda}^{(0)}, \quad h^{(k)} = (u_{\lambda}^{(k)}/c_2),\quad k=0,1,\cdots,n.
\end{equation*}
Furthermore, for $k=0,1,\cdots,n$, we have $u_{\lambda}^{(k)}\left|_{t=0}\right.=\partial_tu_{\lambda}^{(k)}\left|_{t=0}\right.=0$. By virtue of (\ref{eq:GBn}), the identity (\ref{eq:cn}) can be rewritten as
\begin{align*}
&\int_0^T\int_{\partial\Omega}g\Lambda^{(v,1\cdots n)}_{L,N}(h^{(1)},\cdots,h^{(n)})\mathrm{d}s\mathrm{d}t=n\int_0^T\int_{\Omega}\nabla (c_1w)\cdot \nabla \left(\frac{c_2}{c_1^n}\prod\limits_{k=1}^{n}v^{(k)}\right)\mathrm{d}x\mathrm{d}t\\
=&-n\lambda^n\int_{0}^T\int_{\Omega}\left(\prod\limits_{k=0}^{n}\chi^{(k)}\right)\left(\kappa_0\nabla\varphi^{(0)}\cdot \sum\limits_{k=1}^{n}\kappa_k\nabla \varphi^{(k)}\right) c_2 \left(\prod\limits_{k=0}^{n}a^{(k)}\right)e^{\mathrm{i}\lambda \sum\limits_{k=0}^{n}\kappa_k\varphi^{(k)}}\mathrm{d}x\mathrm{d}t+\mathcal{O}(\lambda),
\end{align*}
which in turn implies that
\begin{align}\label{eq:c2 eqn}
&\lambda^{-n}\int_0^T\int_{\partial\Omega}g\Lambda^{(v,1\cdots n)}_{L,N}(h^{(1)},\cdots,h^{(n)})\mathrm{d}s\mathrm{d}t =-n\int_{0}^T\int_{\Omega}e^{\mathrm{i}\lambda \widetilde{\widetilde{\varphi}}}\widetilde{\widetilde{a}} \mathrm{d}x\mathrm{d}t+\mathcal{O}(\lambda^{-1}),
\end{align}
where
\begin{align*}
&\widetilde{\widetilde{\varphi}} = \sum\limits_{k=0}^{n}\kappa_k\varphi^{(k)},\\
&\widetilde{\widetilde{a}}=\left(\prod\limits_{k=0}^{n}\chi^{(k)}\right)\left(\kappa_0\nabla\varphi^{(0)}\cdot \sum\limits_{k=1}^{n}\kappa_k\nabla \varphi^{(k)}\right) c_2 \left(\prod\limits_{k=0}^{n}a^{(k)}\right).
\end{align*}

According to the stationary phase Lemma~\ref{lem:sp lem}, the oscillatory integral in the right hand side of (\ref{eq:c2 eqn}) mainly depends on
\begin{equation*}
	\widetilde{\widetilde{a}}(p)=\left[\left(\prod\limits_{k=0}^{n}\chi^{(k)}\right)\left(\kappa_0\nabla\varphi^{(0)}\cdot \sum\limits_{k=1}^{n}\kappa_k\nabla \varphi^{(k)}\right) c_n \left(\prod\limits_{k=0}^{n}a^{(k)}\right)\right](p),
\end{equation*}
as $\lambda\to+\infty$. It follows from the arbitrariness of $\{\xi^{(k)}\}_{k=0}^{n}$, as well as (\ref{eq:varphi xin}), and (\ref{eq:a xin}) that
\begin{equation*}
	\kappa_0\xi^{(0)}\cdot \left(\sum\limits_{k=1}^{n}\kappa_k\nabla \xi^{(k)}\right) = -\|\kappa_0\xi^{(0)}\|^2\neq 0,
\end{equation*}
which readily yields the uniqueness in the determination of $c_n$.


\section{Proof of Theorem~\ref{thm:IP2}}\label{sec:5}

In the previous section, we have shown that under the zero initial conditions, we can determine the coefficients by measuring $\Lambda_{L,N}$.
In this section, if $L$ is a-priori known we shall demonstrate that under the non-zero initial conditions of $\varphi$ and $\psi$, we can employ the observation inequality of the linear hyperbolic equation to justify that $\Lambda_{\varphi,\psi,L,N}$ can construct the nonlinear coefficient $c_2$ and the initial conditions simultaneously.

Firstly, we present the observation inequality for the inverse problem of the initial value problem \cite{DE17}.
Consider the equation
\begin{equation}
\left\{
\begin{aligned}
	&\rho(x)u_{tt}-\Delta u=0\ &&\text{in}\ M,\\
	&u(x,0) =\varphi,\ u_t(x,0) =\psi\ &&\text {on}\ \Omega,\\
	&u(x,t) = 0\ &&\text {in}\ \Sigma,
\end{aligned}
\right.
\end{equation}
with $0<\rho_1\leq \rho(x)\leq \rho_2$ and $\rho\in C^1(\overline{\Omega})$. Define the energy of $u$ by
\begin{equation*}
E[u](t):=\frac{1}{2}\int_{\Omega}\rho(x)|\partial_tu(x,t)|^2\mathrm{d}x+\frac{1}{2}\int_{\Omega}|\nabla_xu(x,t)|\mathrm{d}x.
\end{equation*}

\begin{lemma}\label{lem:observation inequality} Let $\Gamma_1$ be an open subset of $\partial\Omega$ satisfying
\begin{equation*}
		\{x\in\partial\Omega,\text{such that $x\cdot \nu>0$}\}\subset \Gamma,
\end{equation*}
$R=\sup\{|x|,x\in\Omega\}$, and {\bf Assumption II} hold. Assume that
\begin{equation*}
		 \beta T> 4R\sqrt{\rho_2},
\end{equation*}
where $\beta$ is given in {\bf Assumption II}. Then there exists a constant $C>0$ such that the observability result
\begin{equation*}
		E[u](0)\leq C\int_0^T\int_{\Gamma}|\partial_{\nu}u(x,t)|^2\mathrm{d}x\mathrm{d}t,
\end{equation*}
holds for any solution $u$ with the initial datum $(\varphi,\psi)\in  H_0^1(\Omega)\times L^2(\Omega)$.
\end{lemma}

Recall the first order linearization part of $u$ in (\ref{eq:uj}). For $j=1,2$, we set $u^{(j)}$ to satisfy
\begin{equation*}
\left\{
\begin{aligned}
    &u^{(j)}_{tt} -\Delta (c_1u^{(j)})=0\ &&\text{in}\ M,\\
	&u^{(j)}(0,x) = \varphi^{(j)}(x),\ u^{(j)}_t(0,x) = \psi^{(j)}(x)\ &&\text {on}\ \Omega,\\
	&u^{(j)}(t,x) = h^{(1)}(t,x)\ &&\text {in}\ \Sigma.
\end{aligned}
\right.
\end{equation*}
Let $v^{(j)}=c_1u^{(j)}$. Then
\begin{equation}\label{eq:vjn}
\left\{
\begin{aligned}
&v^{(j)}_{tt}-c_1\Delta v^{(j)}=0\ &&\text{in}\ M,\\
&v^{(j)}(0,x) = c_1\varphi^{(j)}(x),\   v^{(j)}_t(0,x) =c_1 \psi^{(j)}(x)\ &&\text{in}\ \Omega,\\
&v^{(j)}(t,x)=c_1 h^{(1)}(t,x)\ &&\text{on}\ \Sigma,
\end{aligned}
\right.
\end{equation}
and
\begin{equation}\label{eq:Lambda vn}
\Lambda^{(v^{(j)},1)}_{L,N}:c_1h^{(1)}=v^{(j)}\left|_{\Sigma}\right.\mapsto c_1\frac{\partial v^{(j)}}{\partial\nu}.
\end{equation}
Set $v= v^{(1)}-v^{(2)}$. Then $v$ satisfies
\begin{equation}
\left\{
\begin{aligned}
    &v_{tt} - c_1\Delta v =0\ &&\text{in}\ M,\\
	&v(0,x) =c_1( \varphi^{(1)}-\varphi^{(2)}),\  v_t(0,x) = c_1(\psi^{(1)}-\psi^{(2)})\ &&\text {on}\ \Omega,\\
	&v(t,x) = 0\ &&\text {in}\ \Sigma.
\end{aligned}
\right.
\end{equation}
It is noted that $L$ is a-priori given. Then the condition $\Lambda_{\varphi_1,\psi_1,L,N_1}(h^{(1)})=\Lambda_{\varphi_2,\psi_2,L,N_2}(h^{(1)})$ for all $h^{(1)}\in \mathcal{O}^{m+1}(\Sigma)$ implies that
$\Lambda^{(v^{(1)},1)}_{L,N}(h^{(1)})=\Lambda^{(v^{(2)},1)}_{L,N}(h^{(1)})$ for all $h^{(1)}\in \mathcal{O}^{m+1}(\Sigma)$, which together with (\ref{eq:Lambda vn}) and $c_1>0$ readily yields that $\partial_{\nu}v=0$.

Let $w=v$ and $\rho = \frac{1}{c_1}$. Then it can be seen that $w$ satisfies
\begin{equation}
\left\{
\begin{aligned}
&\rho w_{tt}-\Delta w=0\ &&\text{in}\ M\\
&w(x,0) =c_1(\varphi^{(1)}-\varphi^{(2)}),\ w_t(x,0) =c_1(\psi^{(1)}-\psi^{(2)})\ &&\text {on}\ \Omega,\\	
&w(x,t) = 0\ &&\text {in}\ \Sigma,
\end{aligned}
\right.
\end{equation}
with the observation data as
\begin{equation*}
	\partial_{\nu} w = \partial_{\nu} (v) = 0\ \text{on}\ \Sigma.
\end{equation*}
By Lemma\,\ref{lem:observation inequality}, we have
\begin{equation*}
	\frac{1}{2}\int_{\Omega}\rho(x)|\partial_tw(x,0)|^2\, \mathrm{d}x+\frac{1}{2}\int_{\Omega}|\nabla_xw(x,0)|^2\, \mathrm{d}x=0,
\end{equation*}
which implies
\begin{eqnarray}\label{eq:ini a}
\partial_tw(x,0)&=&c_1(x)(\psi^{(1)}(x)-\psi^{(2)}(x))=0,\\
\nabla_xw(x,0) &=& \nabla (c_1(x)(\varphi^{(1)}(x)-\varphi^{(2)}(x)))=0.
\label{eq:ini b}
\end{eqnarray}
For the identity (\ref{eq:ini a}), since $c_1>0$ in $\Omega$, we have  $\psi^{(1)}=\psi^{(2)}$.
The formula (\ref{eq:ini b}) together with the compatibility condition (\ref{eq:compatibility conditions}) yields that
\begin{equation}
	\varphi^{(1)}(x)-\varphi^{(2)}(x)=0 \quad \mbox{on}~\Gamma,
\end{equation}
which immediately gives that $\varphi^{(1)}=\varphi^{(2)}$. Finally, by following a complete similar argument as that in the previous section, we can show that $c_n$ can be uniquely determined.


\section*{Acknowledgment}

The work of H. Liu was supported by the Hong Kong RGC General Research Funds (projects 11311122, 11300821 and 12301420), NSF/RGC Joint Research Fund (project N\_CityU101/21) and the ANR/RGC Joint Research Fund (project A\_CityU203/19). The work of K. Zhang is supported in part by China Natural National Science Foundation (No.~12271207), and by the Fundamental Research Funds for the Central Universities, China.


\end{document}